\documentclass{amsart}
\usepackage[utf8]{inputenc}
\usepackage{mathrsfs}
\usepackage{enumitem}
\usepackage{tikz}
\usepackage{amssymb}
\usepackage{amsfonts}
\usepackage{float}
\usetikzlibrary{arrows.meta,bending,calc,intersections}

\newtheorem{definition}{Definition}[section]
\newtheorem{eg}[definition]{Example}
\newtheorem{problem}[definition]{Open Problem}
\newtheorem{prop}[definition]{Proposition}
\newtheorem{theorem}[definition]{Theorem}
\newtheorem{lemma}[definition]{Lemma}

\newtheorem{corollary}[definition]{Corollary}
\newtheorem{construction}[definition]{Construction}

\title{Finitely and non-finitely related words}
\author{Daniel Glasson}
\address{Mathematical Sciences, RMIT University, Melbourne, 3000, Australia}
\email{glassond@outlook.com}
\begin{document}

\begin{abstract}
An algebra is finitely related (or has finite degree) if its term functions are determined by some
finite set of finitary relations. Nilpotent monoids built from words, via Rees
quotients of free monoids, have been used to exhibit many interesting properties
with respect to the finite basis problem. We show that much of this intriguing
behaviour extends to the world of finite relatedness by using interlocking
patterns called chain, crown, and maelstrom words. In particular, we show that
there are large classes of non-finitely related nilpotent monoids that can be
used to construct examples of: ascending chains of varieties alternating between
finitely and non-finitely related; non-finitely related semigroups whose direct
product are finitely related; the addition of an identity element to a
non-finitely related semigroup to produce a finitely related semigroup.
\end{abstract}
\maketitle
\section{Introduction}
Given a semigroup $\mathbf{S}$, a word $\mathbf{t}$ in $n$ variables induces an $n$-ary \emph{term function} $\mathbf{t}^\mathbf{S}: S^n \to S$ by evaluation. The set of all finitary term functions on a semigroup ~$\mathbf{S}$ is called the \emph{clone} of $\mathbf{S}$, and is denoted $\text{Clo}(\mathbf{S})$. A $k$-ary relation $R \subseteq S$ is \emph{compatible} with $\mathbf{S}$ if it forms a subsemigroup of $\mathbf{S}^k$. By a classical result of universal algebra, any finitary operation that preserves all finitary compatible relations on a finite semigroup $\mathbf{S}$ is a term function on $\mathbf{S}$. In other words, the set of all finitary compatible relations on a finite semigroup determines its clone. If a finite subset of the finitary relations suffices to determine the clone, then $\mathbf{S}$ is said to be \emph{finitely related}.

The study of finitely related semigroups began with Davey, Jackson, Pitkethly and Szab\'{o} \cite{Davey2011}, who showed, among other things, that all finite commutative semigroups/monoids and all finite nilpotent  semigroups are finitely related. Mayr \cite{Mayr2013} built on their work, providing the first example of a non-finitely related finite semigroup. Mayr also showed that Clifford semigroups are finitely related and that finite relatedness is preserved by the addition of a zero element. Dolinka \cite{Dolinka2018} showed that all finite bands are finitely related, answering a question posed by Mayr \cite{Mayr2013}. More recently, Steindl \cite{steindl2022} answered a further question by Mayr, giving an example of a non-finitely related nilpotent monoid, thus establishing that finite relatedness is not preserved by the addition of an identity element. 

Interestingly, the study of finitely related semigroups has developed in much the same way as the study of the finite basis property for semigroups. We continue this trend by beginning the exploration of the finite relatedness of nilpotent monoids built from words. These monoids have previously provided a plethora of interesting examples of finitely and non-finitely based semigroups \cite{jackson2001finite,jackson2018monoid,jackson2000finitely,lee2013finitely,sapir2000finitely}. 

We will build off the work of Steindl \cite{steindl2022} by using nilpotent monoids constructed from words to produce large classes of non-finitely related semigroups. Examples drawn from these classes are used to show some interesting behaviour of the finite relatedness of semigroups. In particular, we show that:
\begin{itemize}
    \item There are no inherently non-finitely related nilpotent monoids (Corollary \ref{cor_inherentnilpotentmonoid});
    \item There exist two non-finitely related semigroups whose direct product is finitely related (Corollary \ref{cor_existencetwonfrisfr});
    \item There exists a non-finitely related semigroup $\mathbf{S}$ such that $\mathbf{S}^1$ is finitely related (Corollary \ref{cor_conformal});
    \item There exists an infinite ascending chain of finitely generated semigroup varieties alternating between finitely and non-finitely related (Corollary \ref{cor_alternatingfrnfr});
    \item There exists a non-finitely based semigroup which is finitely related (Theorem \ref{thm_asabtbfr}).
\end{itemize}

\section{Preliminaries}
\subsection{Finitely related algebras}
An \emph{algebra} $\mathbf{A} = \langle A; F \rangle$ is a set $A$ with a set of finitary operations $F$ on $A$. We refer the reader to Burris and Sankappanavar \cite{Burris1981} for a good introduction to the study of algebras. The set of all finitary term functions on an algebra $\mathbf{A}$ is called the \emph{clone of term functions} of $\mathbf{A}$. We refer to the clone of term functions of an algebra $\mathbf{A}$ as simply the clone of $\mathbf{A}$ and denote it ~$\text{Clo}(\mathbf{A})$. 

Given a set $A$, an operation $f: A^n \to A$ preserves a relation $R \subseteq A^k$ if $f(r_1,\dots,r_n) \in R$ for all $r_1,\dots,r_n \in R$, where $f$ is applied coordinatewise. An algebra $\mathbf{A}$ is said to be \emph{finitely related} if there exists a finite set of finitary relations such that the operations preserving those relations are precisely the term functions of $\mathbf{A}$. We will say that $\mathbf{A}$ is \emph{non-finitely related} if no such set of finitary relations exist. Finitely related algebras have also been called algebras possessing finite degree by some authors \cite{Davey2011,szendrei1986clones}.

Given a function $f:A^n \to A$, we may identify two coordinates, say $i$ and $j$ (with ~$i < j$), to produce $f_{ij}:A^n \to A$ which maps
\[
(x_1,\dots,x_n) \mapsto f(x_1,\dots,x_{i-1},x_j,x_{i+1},\dots,x_n).
\]
We call the function $f_{ij}$ an \emph{identification minor}.

Given a term $\mathbf{t}(x_1,\dots,x_n)$ and $i,j \in \underline{n} = \{1,\dots,n\}$ with $i < j$, we denote the term obtained from $\mathbf{t}$ by replacing each instance of $x_i$ with $x_j$ by $\mathbf{t}^{(ij)}$.

Let $V$ be a variety and $\mathscr{F} = \{\mathbf{t}_{ij}(x_1,\dots,x_n) \mid 1 \leq i < j \leq n\}$ be an indexed family of $n$-ary terms. We say that $\mathscr{F}$ is a \emph{scheme} for $V$ if $\mathscr{F}$ satisfies the following conditions:
\\
Dependency: for all $\mathbf{A} \in V$, $\mathbf{t}_{ij}^\mathbf{A}$ does not depend on its $i$-th coordinate;
\\
Consistency: for all $i,j,k,l \in \mathbb{N}$ with $i < j$ and $k < l$, $V$ satisfies
     \[
     V \models \mathbf{t}_{ij}^{(kl)} \approx \mathbf{t}_{kl}^{(ij)}.
     \]

A scheme $\mathscr{F} = \{\mathbf{t}_{ij}(x_1,\dots,x_n) \mid 1 \leq i < j \leq n\}$ \emph{comes from a term} for $V$ if there exists a term $\mathbf{s}$ such that $V \models \mathbf{s}^{(ij)} \approx \mathbf{t}_{ij}$ for all $1 \leq i < j \leq n$.

\begin{theorem}[{\cite[Theorem 2.9]{Davey2011}}]\label{thm_frconditions}
Let $\mathbf{A}$ be a finite algebra. Then $\mathbf{A}$ is finitely related if and only if there exists $k \geq \lvert A \rvert$ such that the following equivalent conditions hold:
\begin{enumerate}[label=(\arabic*)]
    \item for all $n > k$, every $n$-ary scheme for $V(\mathbf{A})$ comes from a term;
    \item for all $n > k$, if $f_{ij}$ is a term function for every $1 \leq i < j \leq n$ then $f$ is a term function for $\mathbf{A}$. 
\end{enumerate}
\end{theorem}
Davey et al. \cite{Davey2011} and Dolinka \cite{Dolinka2018} used schemes to great effect, while Mayr \cite{Mayr2013} and Steindl \cite{steindl2022} both opted to use operations with minor identified terms. For the sake of convenience, we have chosen to use schemes in this paper as we will repeatedly use the equational theory of monoids built from words.

We briefly note that any $n$-ary scheme for an algebra $\mathbf{A}$ with $n > \lvert A \rvert$ induces a unique operation $f: A^n \to A$ such that $f_{ij} = \mathbf{t}_{ij}^\mathbf{A}$ (see \cite[Lemma 2.6]{Davey2011} for details) by application of the pigeonhole principle.

The \emph{term degree} of an algebra is the least $k \geq \lvert A \rvert$ such that Theorem \ref{thm_frconditions} holds. If we remove the requirement that $k \geq \lvert A \rvert$, then the \emph{degree} of an algebra is the least $k$ such that condition (2) of Theorem \ref{thm_frconditions} holds.

\begin{theorem}[{\cite[Theorem 2.11]{Davey2011}}]\label{thm_frvarietalproperty}
 If two algebras generate the same variety then they are either both finitely or both non-finitely related.   
\end{theorem}

Given a set $A$ and an operation $f:A^n \to A$, we say $f$ depends on its $j$-th coordinate if there exists $a,b \in A^n$ such that $a_i = b_i$ for all $i \in \underline{n}\setminus\{j\}$ and $a_j \neq b_j$ with $f(a) = f(b)$. A scheme depends on all of its variables if the unique operation it induces depends on all of its coordinates. Throughout this paper we will assume that any scheme depends on all of its variables since otherwise a term can simply be found by identification (provided the scheme has sufficient arity) \cite[Remark 2.21]{Davey2011}.

\subsection{Nilpotent monoids and the structure of their schemes.}

A semigroup $\mathbf{S}$ is said to be $d$-\emph{nilpotent} if it satisfies
\[
\mathbf{S} \models x_1\cdots x_d \approx y_1 \cdots y_d.
\]
A \emph{nilpotent monoid} is a nilpotent semigroup with an adjoined identity.  

Davey et al. \cite[Theorem 3.4]{Davey2011} proved that all finite nilpotent semigroups are finitely related. Mayr \cite[Theorem 5.1]{Mayr2013} showed that every finite $3$-nilpotent monoid is finitely related, which, by semigroup lore, implies that almost all finite monoids are finitely related \cite{koubek1985}. Steindl \cite[Theorem 2.5]{steindl2022} gave the first example of a non-finitely related nilpotent monoid, hence showing that finite relatedness is not preserved by adjoining an identity.

For a scheme over a nilpotent monoid (and any semigroup containing a $2$-element semilattice), we may easily recover the fact that all variables, except the identified variable, appear in the terms in the scheme. 

Given a word $\mathbf{w}$, the content of $\mathbf{w}$, denoted $c(\mathbf{w})$, is the set of all variables appearing in $\mathbf{w}$. Throughout this paper, we will set $X_n = \{x_1,\dots,x_n\}$.

\begin{lemma}[{\cite[Lemma 3.2]{steindl2022}}, {\cite[Lemma 3.2]{Dolinka2018}}]\label{lemma_containssemilattice}
Let $\mathbf{S}$ be a semigroup such that ~$V(\mathbf{S})$ contains the variety of semilattices. Then for $n > \lvert S \rvert + 1$, and any $n$-ary scheme $\mathscr{F} = \{\mathbf{t}_{ij} \lvert 1 \leq i < j \leq n\}$, we have $c(\mathbf{t}_{ij}) = X_n \setminus \{x_i\}$. 
\end{lemma}

Let $\mathbf{S}$ be a monoid and $\mathbf{t}(x_1,\dots,x_n)$ be a term of $\mathbf{S}$. Let $Y \subseteq X_n$. We denote the term obtained from $\mathbf{t}$ by deleting $X_n\setminus Y$ by $\mathbf{t}[Y]$. If $f:S^n \to S$ is an operation, then $f[Y]$ is the operation obtained by restricting $f$ to 
\[
\{(a_1,\dots,a_n) \in S^n \mid a_i = 1 \text{ if } x_i \in \{x_1,\dots,x_n\}\setminus Y\}.
\] 

Given a scheme $\mathscr{F} = \{\mathbf{t}_{ij} \lvert 1 \leq i < j \leq n\}$ for a monoid $\mathbf{S}$ that depends on all of its variables (with $n > \lvert S \rvert + 1$), the consistency condition and Lemma \ref{lemma_containssemilattice} ensures that 
\[
\mathbf{S} \models \mathbf{t}_{jk}[x_i] \approx \mathbf{t}_{jk}^{(lm)}[x_i]  \approx x_i^{e_i} \approx x_i^{f_i} \approx \mathbf{t}_{lm}^{(jk)}[x_i] \approx \mathbf{t}_{lm}[x_i]
\]
 for any $j,k,l,m \in \underline{n}\setminus\{i\}$ and some $e_i,f_i \in \mathbb{N}$. 
 
 Let $\mathbf{S}$ be a finite semigroup. The smallest positive integers $m$ and $r$ such that $\mathbf{S} \models x^m \approx x^{m+r}$ are known as the index and period of $\mathbf{S}$ respectively.
 It follows that for a monoid $\mathbf{S}$, the occurrences of variables in the equations of $\mathscr{F}$ must obey the index and period of $\mathbf{S}$ across different terms in that scheme. 
 
 Let $\underline{n} = \{1,\dots,n\}$ and $\text{occ}(x,\mathbf{w})$ be the number of occurrences of a variable $x$ in a word (or term) $\mathbf{w}$.
 
 To each variable $x_i$ in a scheme over $\mathbf{S}$, we may assign 
 \[
 e_i = \text{min}\{\text{occ}(x_i,\mathbf{t}_{jk}) \mid j,k \in \underline{n}\setminus\{i\}\},
 \]
 where $e_i$ is known as the \emph{variable exponent} of $x_i$ in $\mathscr{F}$.

A word $\mathbf{w}$ is an \emph{isoterm} for a semigroup $\mathbf{S}$ if $\mathbf{S} \models \mathbf{u} \approx \mathbf{w} \implies \mathbf{u} = \mathbf{w}$. 
 
 The proof of the following lemma closely follows \cite[Lemma 3.4]{steindl2022}.

 \begin{lemma}\label{lemma_occurences}
     Let $\mathbf{S}$ be a monoid such that $V(\mathbf{S})$ contains the variety of semilattices and $\mathbf{S}$ satisfies $x^{p} \approx x^{p+q}$ for minimal $p,q \in \mathbb{N}$. Let $\mathscr{F} = \{\mathbf{t}_{ij} \lvert 1 \leq i < j \leq n\}$ be a scheme for $\mathbf{S}$ with $n > \lvert S \rvert + 1$ and variable exponents $e_1,\dots,e_n$. Then 
     \begin{enumerate}[label=(\roman*)]
         \item for all $i \in \underline{n}$ and all $j,k \in \underline{n}\setminus\{i\}$ if $e_i < p$ then $\text{occ}(x_i,\mathbf{t}_{jk}) = e_i$ and $\text{occ}(x_i,\mathbf{t}_{jk}) \equiv e_i \pmod{q}$ otherwise;
         \item for all $i,j \in \underline{n}$ if $e_i + e_j < p$ then $\text{occ}(x_j,\mathbf{t}_{ij}) = e_i + e_j$ and $\text{occ}(x_j,\mathbf{t}_{ij}) \equiv e_i + e_j \pmod{q}$ otherwise.
     \end{enumerate}
 \end{lemma}
\begin{proof}
    Suppose for the variable $x_i$ in $\mathscr{F}$ that $e_i < p$. Then there exists $l,m \in \underline{n}\setminus\{i\}$ such that $\mathbf{t}_{lm}[x_i] = x_i^{e_i}$. By the consistency condition, for any $j,k \in \underline{n}\setminus\{i\}$, if $\mathbf{t}_{jk}[x_i] = x_i^f$ then $\mathbf{S} \models x_i^{e_i} \approx x_i^f$. But $e_i$ is less than the index of $\mathbf{S}$, so $x_i^{e_i}$ is an isoterm for $\mathbf{S}$ and hence $f = e_i$. 
    
    Now suppose that $e_i \geq p$. Then by the same argument, if $\mathbf{t}_{jk}[x_i] = x_i^f$ then $\mathbf{S} \models x_i^{e_i} \approx x_i^f$. It follows that $f = \text{occ}(x_i,t_{kl}) \equiv e_i \pmod{q}$. This concludes the proof of (i).

    To prove (ii) suppose $e_i + e_j < p$. Then by (i), and the fact that $n > \lvert S \rvert + 1 \geq 3$, there exists $k,l \in \underline{n}\setminus\{i,j\}$ such that $\text{occ}(x_i,\mathbf{t}_{kl}) = e_i$ and $\text{occ}(x_j,\mathbf{t}_{kl}) = e_j$. It follows that $\text{occ}(x_j,\mathbf{t}_{kl}^{(ij)}) = e_i + e_j$.
    Let $f = \text{occ}(x_j,\mathbf{t}_{ij})$. Then, by the consistency condition, we have 
    \[
    x_j^f = \mathbf{t}_{ij}[x_j] \approx \mathbf{t}_{ij}^{(kl)}[x_j] \approx \mathbf{t}_{kl}^{(ij)}[x_j] = x_j^{e_i+e_j}. 
    \]
    But $e_i + e_j < p$ so $x_j^{e_i+e_j}$ is an isoterm for $\mathbf{S}$, hence $f = e_i + e_j$ as required. 

    If $e_i + e_j > p$, then $f = \text{occ}(x_j,\mathbf{t}_{ij}) \equiv e_i + e_j \pmod q$. This concludes the proof of part (ii) and the proof of the lemma.     
\end{proof}

For a nilpotent monoid, the value of $q$ in Lemma \ref{lemma_occurences} is $1$, and $p$ is the smallest integer such that $\mathbf{S}\setminus\{1\} \models x^p \approx 0$.

Given a word $\mathbf{w}$, a variable $x \in c(\mathbf{w}) \subseteq A$ is said to be \emph{primitive} in $\mathbf{w}$ with respect to a nilpotent monoid $\mathbf{S}$ if under any assignment $\theta:A \to \mathbf{S}$ either $\theta(x) = 1$ or $\theta(\mathbf{w}) = 0$. Given a nilpotent monoid $\mathbf{S}$, we will denote the primitive letters of a word $\mathbf{w}$ with respect to $\mathbf{S}$ by $\text{Prim}(\mathbf{w})$. 

\begin{definition}
Let $\mathbf{S}$ be a monoid satisfying 
\[
\mathscr{A}_{\alpha,\beta} := \{x^\alpha \approx x^{\alpha+\beta}, t_1xt_2x\dots t_\alpha x \approx x^\alpha t_1t_2\dots t_\alpha\}
\]
with $\alpha,\beta \in \mathbb{N}$ minimal with respect to $\mathbf{S} \models \mathscr{A}_{\alpha,\beta}$. Then, for any term $\mathbf{t}$ of $\mathbf{S}$, a variable that appears at least $\alpha$ times in $\mathbf{t}$ is said to be \emph{strongly primitive} in $\mathbf{t}$. 
\end{definition}

We denote $\mathscr{A}_{\alpha,1}$ simply by $\mathscr{A}_{\alpha}$.

Given a monoid $\mathbf{S}$ satisfying $\mathscr{A}_{\alpha,\beta}$, we denote the strongly primitive letters of a word $\mathbf{u}$ by $\text{Prim}_{st}(\mathbf{u})$. If $\mathscr{F} = \{\mathbf{t}_{ij} \mid 1 \leq i < j \leq n\}$ is a scheme for $\mathbf{S}$ then define
\[
\text{Prim}_{st}(\mathscr{F}) = \{x_k \mid \exists i,j \in \underline{n}\setminus\{k\} \text{ such that } x_k \in \text{Prim}_{st}(\mathbf{t}_{ij})\}.
\]
Lemma \ref{lemma_occurences} ensures that if $x_k \in \text{Prim}_{st}(\mathscr{F})$ then $x_k$ is strongly primitive in all terms in $\mathscr{F}$ except the terms which do not depend on $x_k$. 

\begin{lemma}\label{lemma_schemeprimitive}
    Let $\mathbf{S}$ be a monoid satisfying $\mathscr{A}_{\alpha,\beta}$. Let $\mathscr{F} = \{\mathbf{t}_{ij} \lvert 1 \leq i < j \leq n\}$ be a scheme for $\mathbf{S}$ with $n > \lvert S \rvert + 1$ and variable exponents $e_1,\dots,e_n$. Let $f:S^n \to S$ be the operation determined by $\mathscr{F}$. Then 
    \[
    f = f[X_n\setminus\textup{Prim}_{st}(\mathscr{F})] \cdot \prod_{x_k \in \textup{Prim}_{st}(\mathscr{F})}x_k^{e_k}.
    \]
\end{lemma} 
\begin{proof}
    First suppose $V(\mathbf{S})$ does not contain the variety of semilattices. Then $\mathbf{S}$ has no idempotents other than the identity so is a group. It follows from the second law in $\mathscr{A}_{\alpha,\beta}$ that $\mathbf{S}$ is an abelian group whence the lemma's statement holds trivially by the fact $\mathscr{F}$ comes from a term by \cite[Theorem 3.6]{Davey2011}. We may then assume that $V(\mathbf{S})$ contains the variety of semilattices. 
    
    If $\text{Prim}_{st}(\mathscr{F}) = \varnothing$ then there is nothing to prove.
    Now assume there is a single strongly primitive variable in the scheme $\mathscr{F}$. Without loss of generality suppose $\text{Prim}_{st}(\mathscr{F}) = \{x_n\}$. Then for any $(a_1,\dots,a_n) \in S^n$ there exists distinct $i,j \in \underline{n-1}$ such that $a_i = a_j$. Then
\begin{align*}
    f(a_1,\dots,a_{n-1},1)\cdot a_n^{e_n} &= \mathbf{t}_{ij}(a_1,\dots,a_{n-1},1)\cdot a_n^{e_n} \\
    &= \mathbf{t}_{ij}(a_1,\dots,a_n) \\
    &=f(a_1,\dots,a_n),
\end{align*}
by Lemma \ref{lemma_occurences} and the fact that $\mathbf{S} \models \mathscr{A}_{\alpha,\beta}$ with $e_n \geq \alpha$.

Now we consider the remaining case where there exist distinct variables $x_i,x_j \in \text{Prim}_{st}(\mathscr{F})$. We claim that $\mathscr{F}$ comes from the term 
\[
\mathbf{w} = \mathbf{t}_{ij}[X_n\setminus\text{Prim}_{st}(\mathscr{F})]\cdot \mathbf{s},
\]
where $\mathbf{s} = \prod_{x_m \in \text{Prim}_{st}(\mathscr{F})}x_m^{e_m}$.
Take any distinct $k,l \in \underline{n}$. First suppose that $e_k + e_l < \alpha$ implying that neither $x_k$ nor $x_l$ is strongly primitive. Then 
\begin{align*}
    (\mathbf{t}_{ij}[X_n\setminus\text{Prim}_{st}(\mathscr{F})])^{(kl)} &\approx \mathbf{t}_{ij}^{(kl)}[X_n\setminus\text{Prim}_{st}(\mathscr{F})] \\
    &\approx \mathbf{t}_{kl}^{(ij)}[X_n\setminus\text{Prim}_{st}(\mathscr{F})] \\
    &\approx \mathbf{t}_{kl}[X_n\setminus\text{Prim}_{st}(\mathscr{F})],
\end{align*}
since $x_i,x_j \in \text{Prim}_{st}(\mathscr{F})$. It follows that 
\[
\mathbf{w}^{(kl)} \approx (\mathbf{t}_{ij}[X_n\setminus\text{Prim}_{st}(\mathscr{F})])^{(kl)}\cdot \mathbf{s}^{(kl)} \approx \mathbf{t}_{kl}[X_n\setminus\text{Prim}_{st}(\mathscr{F})]\cdot \mathbf{s}^{(kl)},
\]
and thus $\mathbf{w}^{(kl)} \approx \mathbf{t}_{kl}$.

Now suppose $e_k + e_l \geq \alpha$ and let $Y = \text{Prim}_{st}(\mathscr{F})\cup\{x_k,x_l\}$. Then, by the consistency of $\mathscr{F}$ and the fact that $x_i,x_j,x_k,x_l \in Y$, we have
\begin{equation}\label{eqn_primscheme1}
(\mathbf{t}_{kl}[X_n\setminus Y])^{(ij)} \approx \mathbf{t}_{kl}[X_n\setminus Y] \approx \mathbf{t}_{ij}[X_n\setminus Y] \approx (\mathbf{t}_{ij}[X_n\setminus Y])^{(kl)}.    
\end{equation}

Let $p = \text{occ}(x_l,(\mathbf{t}_{ij}[X_n\setminus\text{Prim}_{st}(\mathscr{F})])^{(kl)})$ and notice
\[
p = \begin{cases}
    e_k + e_l &\text{if } x_k\not\in \text{Prim}_{st}(\mathscr{F}), x_l \not\in \text{Prim}_{st}(\mathscr{F}), \\
    e_k &\text{if } x_k\not\in \text{Prim}_{st}(\mathscr{F}), x_l \in \text{Prim}_{st}(\mathscr{F}), \\
    e_l &\text{if } x_k\in \text{Prim}_{st}(\mathscr{F}), x_l \not\in \text{Prim}_{st}(\mathscr{F}), \\
    0 &\text{if } x_k\in \text{Prim}_{st}(\mathscr{F}), x_l \in \text{Prim}_{st}(\mathscr{F}).
\end{cases}
\]
Then 
\[
\mathbf{w}^{(kl)} \approx (\mathbf{t}_{ij}[X_n\setminus\text{Prim}_{st}(\mathscr{F})])^{(kl)}\cdot \mathbf{s}^{(kl)} \approx (\mathbf{t}_{ij}[X_n\setminus Y])^{(kl)}\cdot x_l^p \cdot \mathbf{s}^{(kl)},  
\]
where $x_l^0$ is the empty word. 

By (\ref{eqn_primscheme1}) we get 
\[
(\mathbf{t}_{ij}[X_n\setminus Y])^{(kl)}\cdot x_l^p \cdot \mathbf{s}^{(kl)} \approx \mathbf{t}_{kl}[X_n\setminus Y]\cdot x_l^p \cdot \mathbf{s}^{(kl)}.
\]
But $e_k + e_l \geq \alpha$ and thus $\mathbf{t}_{kl} \approx \mathbf{t}_{kl}[X_n\setminus Y]\cdot x_l^p \cdot \mathbf{s}^{(kl)}$. Therefore $\mathbf{w}^{(kl)} \approx \mathbf{t}_{kl}$ as required.

Finally, since $f$ is a term function and $\mathbf{S} \models \mathscr{A}_{\alpha,\beta}$ we obtain
\[
    f = f[X_n\setminus\text{Prim}_{st}(\mathscr{F})]\cdot \prod_{x_k \in \text{Prim}_{st}(\mathscr{F})}x_k^{e_k}.\qedhere
\]
\end{proof}

\section{Nilpotent monoids built from words.}

Let $A$ be a set and let $W$ be a set of words in the free monoid $A^*$. Let $M(W)$ be the Rees quotient $A^*/I(W)$ where $I(W)$ is the ideal consisting of all words in $A^*$ which are not subwords of words in $W$. It is easy to see that $M(W)$ is a nilpotent monoid whose universe consists of all subwords of words in $W$ (along with the empty word and $0$) and whose binary operation $\cdot$ is defined by 
\[
    u \cdot v = 
    \begin{cases}
    uv & \text{ if $uv$ is a subword of a word in $W$, } \\
    0 & \text{ otherwise.}
    \end{cases}
\]

This construction, usually attributed to Dilworth \cite{morse1944unending}, was employed by Perkins to show the existence of non-finitely based semigroups \cite{Perkins1969}. Since then, constructing nilpotent monoids from words has proven to be instrumental in the exploration of the finite basis problem for semigroups. In addition, this construction has been used to show that the finite basis property is not preserved by direct products \cite{jackson2000finitely} nor by adjoining an identity element \cite{lee2013finitely,Perkins1969}.

Let $\kappa \in \mathbb{N}$. A word $\mathbf{w}$ is \emph{$\kappa$-limited} if $\text{occ}(x,\mathbf{w}) \leq \kappa$ for every letter $x \in c(\mathbf{w})$. Let $A$ be a non-empty set. We denote the set of $\kappa$-limited words in $A^*$ by $A_\kappa$.

The next result concerning $M(A_\kappa)$ is reminiscent of Steindl’s \cite[Theorem 2.4]{steindl2022} for the free nilpotent monoid, though was obtained by the author independently and prior to the posting of \cite{steindl2022}. 

\begin{prop}\label{prop_mak}
    For any words $\mathbf{u}$ and $\mathbf{v}$, $M(A_\kappa) \models \mathbf{u} \approx \mathbf{v}$ if and only if $\textup{Prim}_{st}(\mathbf{u}) = \textup{Prim}_{st}(\mathbf{v})$ and $\mathbf{u}[c(\mathbf{u})\setminus \textup{Prim}_{st}(\mathbf{u})] = \mathbf{v}[c(\mathbf{v})\setminus \textup{Prim}_{st}(\mathbf{v})]$.
\end{prop}
\begin{proof}
Follows immediately from the property that any term for $M(A_\kappa)$ without strongly primitive variables is an isoterm. 
\end{proof}

\begin{theorem}\label{thm_makfinitelyrelated}
Let $A$ be a finite set. Then $M(A_\kappa)$ is finitely related with term degree at most $\text{max}(\lvert M(A_\kappa) \rvert + 1, 4)$.
\end{theorem}
\begin{proof}
We may assume that $\lvert A \rvert \geq 2$ since if $\lvert A \rvert = 1$ then $M(A_\kappa)$ is commutative and finitely related by \cite[Theorem 3.6]{Davey2011} with term degree at most $\text{max}(\lvert M(A_\kappa) \rvert,4)$.

Assume $n > \lvert M(A_\kappa) \rvert + 1$ and let $\mathscr{F} = \{\mathbf{t}_{ij} \mid 1 \leq i < j \leq n\}$ be a scheme for $M(A_\kappa)$ with variable exponents $e_1,\dots,e_n$ that determines an operation $f:M(A_\kappa)^n \to M(A_\kappa)$. 

Let $Y = X_n \setminus \text{Prim}_{st}(\mathscr{F})$. Define a set of new variables 
\[
    \overline{Y} = \{{}_px_i \mid x_i \in Y, 1 \leq p \leq e_i\}.
\]
Since $Y$ contains only non-strongly primitive variables, if $x_i,x_j \in Y$, and $f[x_i,x_j]$ is induced by a term $\mathbf{s}$ then $\mathbf{s}$ is an isoterm for $M(A_\kappa)$ and hence is the unique term inducing $f[x_i,x_j]$. With this in mind, define $\prec$ on $\overline{Y}$ by ${}_px_i \prec {}_qx_j$ if the $p$-th occurrence of $x_i$ occurs before the $q$-th occurrence of $x_j$ in the term inducing $f[x_i,x_j]$. Since each term $f[x_i,x_j]$ is an isoterm, $\prec$ is anti-symmetric and connected. Transitivity is established by the fact that $n > \lvert M(A_\kappa) \rvert + 1 \geq 7$ and so $f[x_i,x_j,x_l]$ is a term function (and, in fact, an isoterm) for any $x_i,x_j,x_l \in Y$. Therefore $\prec$ is a strict linear order on $\overline{Y}$.

By Lemma \ref{lemma_schemeprimitive}, it is sufficient to show that $f[Y]$ comes from a term. Let 
\[
x_{i_1} \prec x_{i_2} \prec \cdots \prec x_{i_\alpha}
\]
be the linear order $\prec$ on $\overline{Y}$ with the prefixes removed so that $i_1,i_2,\dots,i_\alpha$ may have repeats. We claim that $f[Y]$ is induced by the term
\[
    \mathbf{w} = x_{i_1}x_{i_2}\cdots x_{i_\alpha}.
\]
To do this, we will show that $M(A_\kappa)$ satisfies $\mathbf{w}^{(ij)} \approx \mathbf{t}_{ij}[Y]$.

Note, only $x_j$ can be strongly primitive in $\mathbf{w}^{(ij)}$ or $\mathbf{t}_{ij}[Y]$. If $x_j$ is strongly primitive in the former (resp.~latter), then $x_j$ is strongly primitive in the latter (resp.~former) by Lemma \ref{lemma_occurences} and the construction of $\mathbf{w}$. Similarly, if $x_j$ is not strongly primitive in $\mathbf{w}^{(ij)}$ then it is not strongly primitive in $\mathbf{t}_{ij}$ and vice-versa. Therefore ~$\text{Prim}_{st}(\mathbf{w}^{(ij)}) = \text{Prim}_{st}(\mathbf{t}_{ij})$.

Note, for all ~$x_k,x_l \in Y\setminus\{x_i,x_j\}$ we get
\begin{equation}\label{eqn_orderofvariables}
    \mathbf{t}_{ij}[x_k,x_l] = f[x_k,x_l] = \mathbf{w}_{ij}[x_k,x_l].
\end{equation}
Hence if $x_j$ is strongly primitive in $\mathbf{w}^{(ij)}$, and thus $\mathbf{t}_{ij}$, then $\mathbf{w}^{(ij)} \approx \mathbf{t}_{ij}$ by Proposition \ref{prop_mak}. 

Suppose $x_j$ is not strongly primitive in $\mathbf{w}^{(ij)}$. Equation (\ref{eqn_orderofvariables}) implies that ${}_px_k$ occurs before ${}_qx_l$ in $\mathbf{w}^{(ij)}$ if and only if ${}_px_k$ occurs before ${}_qx_l$ in $\mathbf{t}_{ij}$ for any $x_k,x_l \in Y\setminus\{x_i,x_j\}$. It follows that to show $\mathbf{w}^{(ij)} \approx \mathbf{t}_{ij}$, we need to show that ${}_px_j$ occurs before (resp.~after) ${}_qx_k$ in $\mathbf{w}^{(ij)}$ if and only if ${}_px_j$ occurs before (resp.~`after) ${}_qx_k$ in ~$\mathbf{t}_{ij}$. 

Assume ${}_px_j$ occurs before ${}_qx_k$ in $\mathbf{w}^{(ij)}$. Then there exists integers $r$ and $s$ such that $r + s = p$ and ${}_rx_i$ and ${}_sx_j$ occur before ${}_qx_k$ in $\mathbf{w}$. By the construction of $\mathbf{w}$, the operation $f[x_i,x_j,x_k]$ is induced by a unique term in which ${}_rx_i$ and ${}_sx_j$ occur before ${}_qx_k$. Moreover, since $x_j$ is not primitive in $\mathbf{w}^{(ij)}$ (and $\mathbf{t}_{ij}$), $f_{ij}[x_i,x_j,x_k]$ is induced by a unique term where ${}_{r+s}x_j = {}_px_j$ occurs before ${}_qx_k$. But $f_{ij}[x_i,x_j,x_k]$ is induced by $\mathbf{t}_{ij}[x_j,x_k]$ and thus ${}_px_j$ occurs before ${}_qx_k$ in $\mathbf{t}_{ij}$ as required. The converse is similar as is the case when ${}_px_j$ occurs after ${}_qx_k$ in $\mathbf{w}^{(ij)}$ or $\mathbf{t}_{ij}$.  

Therefore $M(A_\kappa) \models \mathbf{w}^{(ij)} \approx \mathbf{t}_{ij}[Y]$ by Proposition \ref{prop_mak} and the theorem is proved by Lemma \ref{lemma_schemeprimitive}.
\end{proof}

Mayr \cite{Mayr2013} proposed a more general way of establishing large classes of non-finitely related algebras: prove there exists an inherently non-finitely related semigroup. An algebra $\mathbf{A}$ is said to be \emph{inherently non-finitely related} if for any algebra $\mathbf{B}$ of the same type, $\mathbf{A} \in V(\mathbf{B})$ implies that $\mathbf{B}$ is non-finitely related. It is unknown whether there exist inherently non-finitely related algebras, however Theorem \ref{thm_makfinitelyrelated} allows us to conclude that the class of nilpotent monoids will not produce an example of one. 

\begin{corollary}\label{cor_inherentnilpotentmonoid}
There are no inherently non-finitely related nilpotent monoids.
\end{corollary}
\begin{proof}
Let $\mathbf{S}$ be a non-finitely related $d$-nilpotent monoid. We claim that $\mathbf{S}$ embeds into a finitely related nilpotent monoid. 

Since $\mathbf{S}$ is $d$-nilpotent, $\mathbf{S} \models \mathscr{A}_d$. Let $A$ be a set with $\lvert A \rvert = 2$ and consider the semigroup $\mathbf{T} := \mathbf{S} \times M(A_{d-1})$. Then $V(\mathbf{T}) = V(M(A_{d-1}))$ \cite[Corollary 3.7]{jackson2000finitely} so is finitely related by Theorem \ref{thm_makfinitelyrelated} and Theorem \ref{thm_frvarietalproperty}. This concludes the proof of the claim as $\mathbf{S}$ embeds into $\mathbf{T}$. Therefore there are no inherently non-finitely related nilpotent monoids.
\end{proof}

\section{Chain, Crown \& Maelstrom words}

This section is primarily dedicated to the construction of schemes using chain, crown, and maelstrom words. Chain words, to be defined precisely below, were initially introduced by Lee for constructing non-finitely based monoid varieties \cite{lee2015inherently,lee2014certain}, and were further studied by Jackson and Lee to analyse monoids with uncountably many subvarieties \cite{jackson2018monoid}. Ren, Jackson, Zhao, and Lei also utilized chain words in the identification of semiring limit varieties \cite{ren2022flat}. Although not explicitly stated, chain words were used by Steindl to give an example of a non-finitely related nilpotent monoid \cite{steindl2022}. For certain varieties, chain words are essentially fixed in an overlapping pattern until a variable is deleted, in which case the pattern is irrevocably damaged. In the case of schemes, we will show that, under certain circumstances, identifying variables to create a strongly primitive variable has the same effect. This will allow us to construct patterns that don't correspond to any term for some variety, but where identifications can be derived from a term.

We will extend the concept of variable patterns that break down upon removal to two novel types of words known as Crown and Maelstrom words. This extension enables us to demonstrate the presence of an abundance of non-finitely related nilpotent monoids, which we can utilize to prove, among other things, the existence of two non-finitely related algebras whose direct product is finitely related.

\subsection{Schemes built from chain words}
The base case of the following definition, where $p = q = 1$, is the most commonly invoked form, and are sometimes referred to as Lee words \cite{ren2022flat}.
 
\begin{definition}\label{defn_chainwords}
Given $n,p,q \in \mathbb{N}$, the \emph{chain word} in $n$ variables with variable exponent $p+q$ is given by
\[
    \mathcal{C}_{n,p,q} = x_1^px_2^px_1^qx_3^px_2^qx_4^px_3^q\cdots x_n^px_{n-1}^qx_n^q.
\]
\end{definition}

Let $i,j,n \in \mathbb{N}$. To simplify the exposition in this section, the following notation is used:

\begin{align*}
    \mathcal{C}_{p,q}(i;j) &:= \mathcal{C}_{j-i+1,p,q}(x_i,x_{i+1},\dots,x_j) &&\text{ if } i < j \leq n, \\
    \mathcal{C}_{p,q}(i;n;j) &:= \mathcal{C}_{n-i+j+1,p,q}(x_i,\dots,x_{n-1},x_n,x_1,\dots,x_j) &&\text{ if } j < i \leq n.\\
\end{align*}

We adopt the convention that chain words are over the usual variable set $X_n$, and denote the chain word over $Y_n = \{y_1,\dots,y_n\}$ by $\overline{\mathcal{C}_{n,p,q}}$.

Given any $i,j \in \underline{n}$ with $i < j$, let 
\[
\theta_{ij} =(x_j,x_{i+1},x_{i+2},\dots,x_{n-1},x_n,x_1,x_2,\dots,x_{i-1}).
\]

\begin{lemma}\label{lemma_chainwordcommute}
Let $\mathbf{S}$ be a monoid. Fix $p,q \in \mathbb{N}$ and suppose 
\[
\mathbf{S} \models \mathcal{C}_{n,p,q}\overline{\mathcal{C}_{m,n,q}} \approx \overline{\mathcal{C}_{m,n,q}}\mathcal{C}_{n,p,q}
\]for all $m,n \in \mathbb{N}$. Then for any $n \in \mathbb{N}$ with $\mathbf{t} = \mathcal{C}_{n,p,q}$, $\mathbf{S}$ satisfies
\begin{equation}\label{eqn_chainwordscommute}
       \mathbf{t}(\theta_{ij})[X_n\setminus\{x_j,x_k,x_l\}] \approx \mathbf{t}(\theta_{kl})[X_n\setminus\{x_i,x_j,x_l\}],
\end{equation}
for all $i,j,k,l \in \underline{n}$ with $i < j$ and $k < l$. 
\end{lemma}
\begin{proof}
Let $\mathbf{u},\mathbf{v}$ be the words on the left and right hand side of (\ref{eqn_chainwordscommute}) respectively. It is easy to see that $c(\mathbf{u}) = c(\mathbf{v})$. Notice that assigning $1$ to any variable in a chain breaks the word into smaller chain words. It follows from the fact that $\mathbf{S} \models \mathcal{C}_{n,p,q}\overline{\mathcal{C}_{m,n,q}} \approx \overline{\mathcal{C}_{m,n,q}}\mathcal{C}_{n,p,q}$ that it is enough to show that $\mathbf{u}$ and $\mathbf{v}$ are products of the same chain words. 

We give the case for when $i < k < j < l$, but it is easy to verify, using the same calculation, that all cases produce the same chain words in $\mathbf{u}$ and $\mathbf{v}$. 

Given that $i < k < j < l$, direct calculation yields
\begin{align*}
    \mathbf{u} &= \mathcal{C}_{p,q}(i+1;k-1)\mathcal{C}_{p,q}(k+1;j-1)\mathcal{C}_{p,q}(j+1;l-1)\mathcal{C}_{p,q}(l+1;n;i-1)\\
    &\approx \mathcal{C}_{p,q}(k+1;j-1)\mathcal{C}_{p,q}(j+1;l-1)\mathcal{C}_{p,q}(l+1;n;i-1)\mathcal{C}_{p,q}(i+1;k-1) = \mathbf{v}.
\end{align*}

\end{proof}

Our main goal is to show that monoids satisfying certain conditions are necessarily non-finitely related. Lemma \ref{lemma_schemeprimitive} and Lemma \ref{lemma_chainwordcommute} hint towards the scheme we are trying to construct. The following lemma will essentially guarantee that these schemes do not come from terms. 

\begin{figure}[H]
    \centering
\begin{tikzpicture}
\begin{scope}[xscale=-1]
\def \n {5}
\def \radius {3cm}
\def \margin {8}

\foreach \s in {1,...,3}
{
  \node[circle] at ({360/\n * (\s - 1)}:\radius) {$x_\s$};
  \draw[->, >=latex] ({360/\n * (\s - 1)+\margin}:\radius) 
    arc ({360/\n * (\s - 1)+\margin}:{360/\n * (\s)-\margin}:\radius);
}
\node[circle] at ({360/\n * (4 - 1)}:\radius) {\rotatebox{55}{$\dots$}};
\draw[->, >=latex] ({360/\n * (4 - 1)+\margin}:\radius) 
    arc ({360/\n * (4 - 1)+\margin}:{360/\n * (4)-\margin}:\radius);
\node[circle] at ({360/\n * (5 - 1)}:\radius) {$x_{n}$};
\draw[->, >=latex] ({360/\n * (5 - 1)+\margin}:\radius) 
    arc ({360/\n * (5 - 1)+\margin}:{360/\n * (5)-\margin}:\radius);
\end{scope}
\end{tikzpicture}
    \caption{A directed $n$-cycle is an impossible pattern for chain words.}
    \label{fig_chainword}
\end{figure}

The scheme we produce will be derived from the directed $n$-cycle in Figure \ref{fig_chainword}. Each vertex represents a variable, and each edge a pattern of occurrence of variables in a desired term. In particular, we require that there is an edge $x \to y$ if and only if $x^py^px^qy^q$ occurs in the term we are trying to construct. The following lemma shows that this is impossible for the graph in Figure \ref{fig_chainword} if $\mathcal{C}_{2,p,q} = x_1^px_2^px_1^qx_2^q$ is an isoterm.

\begin{lemma}\label{lemma_chainnotterm}
Let $n \geq 4$ and $G = (V,E)$ be the directed graph in Figure \ref{fig_chainword}. Fix $p,q \in \mathbb{N}$. Then there is no $n$-ary word $\mathbf{t}$ such that if $(x,y) \in E(G)$ then $\mathbf{t}[x,y] = x^py^px^qy^q$.
\end{lemma}
\begin{proof}
For any $x_i$ we have  $(x_{i\ominus1},x_i) \in E$ where $\ominus$ is subtraction modulo $n$. Suppose $\mathbf{t}$ starts with some variable $x_j$. Then $\mathbf{t}[x_{j\ominus1},x_j] = x_{j\ominus1}^px_j^px_{j\ominus1}^qx_j^q$ contradicting the fact that $\mathbf{t}$ begins with $x_j$. 
\end{proof}

\begin{theorem}\label{thm_chainwordsnfr}
Let $\mathbf{S}$ be a finite monoid. Suppose there exists $p,q \in \mathbb{N}$ such that $\mathbf{S}$ satisfies the following:
\begin{enumerate}[label=(\roman*)]
    \item for any $m,n\in \mathbb{N}$, $\mathbf{S} \models \mathcal{C}_{n,p,q}\overline{\mathcal{C}_{m,p,q}} \approx \overline{\mathcal{C}_{m,p,q}}\mathcal{C}_{n,p,q}$;
    \item $\mathbf{S} \models \mathscr{A}_{\alpha,\beta}$ where $\alpha,\beta \in \mathbb{N}$ and $\alpha \leq 2p+2q$;
    \item $\mathcal{C}_{2,p,q}$ is an isoterm for $\mathbf{S}$.
\end{enumerate}
Then $\mathbf{S}$ is non-finitely related.
\end{theorem}
\begin{proof}
Let $n > 4$. For any $1 \leq i < j \leq n$, let 
\[
    \mathbf{t}_{ij} = \mathcal{C}_{n,p,q}(\theta_{ij}).
\]
Let $\mathscr{F} = \{\mathbf{t}_{ij} \mid 1 \leq i < j \leq n\}$. We first show that $\mathscr{F}$ is a scheme for $V(\mathbf{S})$. Clearly ~$\mathscr{F}$ satisfies the dependency condition since $c(\mathbf{t}_{ij}) = X_n\setminus\{x_i\}$. To show consistency, take any $i,j,k,l \in \underline{n}$ with $i < j$ and $k < l$. Then 
\begin{equation}\label{eqn_thmabab}
    \mathbf{t}_{ij}^{(kl)}[X_n \setminus \{x_j,x_l\}] \approx \mathbf{t}_{kl}^{(ij)}[X_n \setminus \{x_j,x_l\}]
\end{equation}
by Lemma \ref{lemma_chainwordcommute} and (i). Now, (ii) ensures that $x_j,x_l$ are the strongly primitive letters in $\mathbf{t}_{ij}^{(kl)}$ and $\mathbf{t}_{kl}^{(ij)}$ with $\text{occ}(x_j,\mathbf{t}_{ij}^{(kl)}) = 2p+2q = \text{occ}(x_j,\mathbf{t}_{kl}^{(ij)})$ and $\text{occ}(x_l,\mathbf{t}_{ij}^{(kl)}) = 2p+2q = \text{occ}(x_l,\mathbf{t}_{kl}^{(ij)})$ whence $\mathbf{S} \models \mathbf{t}_{ij}^{(kl)} \approx \mathbf{t}_{kl}^{(ij)}$ by (\ref{eqn_thmabab}). Therefore ~$\mathscr{F}$ is a scheme for $V(\mathbf{S})$.

Let $G$ be the graph in Figure \ref{fig_chainword}. By the construction of the terms in $\mathscr{F}$ and (iii), if $\mathscr{F}$ came from a term $\mathbf{t}$ then we would have $\mathbf{t}[x,y] = x^py^px^qy^q$ whenever $(x,y) \in E(G)$. But by Lemma \ref{lemma_chainnotterm} no such term can exist. Therefore $\mathscr{F}$ does not come from a term and hence $\mathbf{S}$ is non-finitely related.
\end{proof}

An \emph{image} of a word ${\bf w}\in A^*$ is a word of the form $\theta({\bf w})$, where $\theta$ is an endomorphism of $A^*$.

It is well known that two words $\mathbf{u}$ and $\mathbf{v}$ in the free monoid $A^*$ commute if and only if there exists a word $\mathbf{w}$ such that $\mathbf{u}$ and $\mathbf{v}$ are powers of $\mathbf{w}$ (see \cite[Lemma ~2.2]{sapir2000finitely} for instance).

\begin{corollary}\label{cor_corchain}
    Let $\mathbf{w}$ be a word in $\{a,b\}^+$. Suppose the following conditions hold:
    \begin{enumerate}[label=(\arabic*)]
    \item there exist subwords $\mathbf{u},\mathbf{v}$ of $\mathbf{w}$ such that $\mathbf{u}$ and $\mathbf{v}$ do not commute in $\{a,b\}^*$ and $\mathbf{u}^p\mathbf{v}^p\mathbf{u}^q\mathbf{v}^q$ is a subword of $\mathbf{w}$;
    \item the image of two chain words are not adjacent in ${\bf w}$ unless the images are powers of the same word;
    \item $M(\mathbf{w})$ satisfies $\mathscr{A}_\alpha$ for $\alpha \leq 2p+2q$.
    \end{enumerate}

    Then $M(\mathbf{w})$ is non-finitely related.
\end{corollary}
\begin{proof}
    The monoid $M(\mathbf{w})$ satisfies conditions (ii) and (iii) in Theorem \ref{thm_chainwordsnfr} by (1) and (3). To show (i) of Theorem \ref{thm_chainwordsnfr} let $m,n \in \mathbf{N}$ and consider $\mathcal{C}_{n,p,q}$ and $\overline{\mathcal{C}_{m,p,q}}$ in variables $x_1,\dots,x_n$ and $y_1,\dots,y_m$ respectively. We may interpret $\mathcal{C}_{n,p,q}$ and ~$\overline{\mathcal{C}_{m,p,q}}$ as $(m+n)$-ary words in the obvious way. Then for any $x \in M(\mathbf{w})^{m+n}$, condition ~(2) implies $(\mathcal{C}_{n,q,p}\overline{\mathcal{C}_{m,p,q}})(x) = 0$ unless $\mathcal{C}_{n,q,p}(x) = 1$ or $\overline{\mathcal{C}_{m,p,q}}(x) = 1$. Thus $M(\mathbf{w})$ satisfies (i) of Theorem \ref{thm_chainwordsnfr}.
\end{proof}

Although Corollary \ref{cor_corchain} offers an abundance of examples of non-finitely related nilpotent monoids, it is not comprehensive in terms of words which meet the requirements of Theorem \ref{thm_chainwordsnfr}. In our example of two non-finitely related semigroups whose direct product is finitely related, we will use the fact that $M(a^pb^pa^qb^q)$ is non-finitely related (as a consequence of Corollary \ref{cor_corchain}). The smallest monoid of this form, when $p = q = 1$, is in fact the smallest non-finitely related nilpotent monoid built from a set of words. 

\begin{eg}
The monoid $M(abab)$ is the smallest example of a non-finitely related nilpotent monoid built from a set of words.
\end{eg}
\begin{proof}
    The 9 element monoid $M(abab)$ is non-finitely related by Corollary \ref{cor_corchain}. The only words $\mathbf{w}$ such that $\lvert M(\mathbf{w}) \rvert < 10$ are the words which have length $3$ or less or the words $aaaa$ and $abab$ (up to letter change) \cite[Theorem 4.3]{jackson2000finitely}. If $\mathbf{w}$ has length $3$ or less then $M(\mathbf{w})$ is $d$-nilpotent for $d \leq 4$. If $d = 3$ or $d = 4$ then $M(\mathbf{w})$ is finitely related by \cite[Theorem 5.1]{Mayr2013} and \cite[Theorem 2.3]{steindl2022} respectively. If $d = 2$ or if $\mathbf{w} = aaaa$ then $M(\mathbf{w})$ is a commutative monoid so is finitely related by \cite[Theorem 3.6]{Davey2011}. Therefore $M(abab)$ is the smallest non-finitely related nilpotent monoid built from words.
\end{proof}

While our main focus in this paper is nilpotent monoids built from words, Theorem \ref{thm_chainwordsnfr} also applies to semigroups that are not nilpotent monoids.
An \emph{ideal extension} of a semigroup $\mathbf{S}$ by a semigroup with zero element $\mathbf{T}$ is a semigroup $\mathbf{U}$ such that $\mathbf{S}$ is an ideal of $\mathbf{U}$ and the Rees quotient $\mathbf{U}/\mathbf{S}$ is isomorphic to $\mathbf{T}$. If $\mathbf{T}$ is a nilpotent semigroup then $\mathbf{U}$ is called an \emph{nilpotent-extension} of $\mathbf{S}$ by $\mathbf{T}$. 

Given a set of words $W$, we obtain $S(W)$ from $M(W)$ by removing the empty word, that is the identity, from $M(W)$.  

\begin{eg}
Let $p,q \in \mathbb{N}$. Let $\mathbf{T} = S(a^pb^pa^qb^q)$ and $\mathbf{S}$ be a commutative monoid with index $m \in \mathbb{N}$ such that $m \leq 2p+2q - (\text{max}(p,q)+1)$. If $\mathbf{N}$ is a nilpotent-extension of $\mathbf{S}$ by $\mathbf{T}$ then $\mathbf{N}^1$ is non-finitely related. 
\end{eg}
\begin{proof}
Since $\mathbf{w} = a^pb^pa^qb^q$ satisfies condition (2) of Corollary \ref{cor_corchain}, it follows that chain words commute in $\mathbf{T}^1$ as in the proof of Corollary \ref{cor_corchain}. Together with the fact that $\mathbf{S}$ is commutative and $ns = sn$ for any $n \in N^1$ and $s \in S$, it follows that $\mathbf{N}^1$ satisfies condition (i) of Theorem \ref{thm_chainwordsnfr}. Condition (ii) of Theorem \ref{thm_chainwordsnfr} holds by the fact that $\mathbf{T}^1 \models \mathscr{A}_{\text{max}(p,q)+1,r}$, where $r$ is the period of $\mathbf{S}$, and $\mathbf{S}$ has an index of no more than $2p+2q - (\text{max}(p,q)-1)$. Finally, condition (iii) holds by the fact that $x^py^px^qy^q$ is an isoterm for $\mathbf{T}^1$ and so is an isoterm for $\mathbf{N}^1$. 
\end{proof}

Note, in the above example we may remove the condition that $\mathbf{S}$ has index $m \leq 2p+2q$ and still use the terms $\mathbf{t}_{ij}$ in Theorem \ref{thm_chainwordsnfr}. In this case $\mathbf{N}^1$ does not satisfy the requirements of Theorem \ref{thm_chainwordsnfr}, but the reader can verify that the family of terms $\mathscr{F} = \{\mathbf{t}_{ij} \mid 1 \leq i < j \leq n\}$ is indeed a scheme for $V(\mathbf{N}^1)$ and that it does not come from a term.  

The following corollary mimics a result in the finite basis world \cite{jackson2000finitely}. 

\begin{corollary}\label{cor_alternatingfrnfr}
Let $A = \{a,b\}$ and $A_\kappa$ be the words over $A^*$ with at most $\kappa$-occurrences of any letter. Set $B_k = A_\kappa \cup \{a^\kappa b^\kappa ab\}$. Then 
    \[
        V(M(A_2)) \subset V(M(B_2)) \subset V(M(A_3)) 
        \subset V(M(B_3)) \subset \cdots
    \]
is an ascending chain of varieties alternating between finitely related and non-finitely related.
\end{corollary}
\begin{proof}
Let $i \geq 2$. The inclusions $V(M(A_i)) \subset V(M(B_i)) \subset V(M(A_{i+1}))$ are obvious. Each $V(M(A_i))$ is finitely related by Theorem \ref{thm_makfinitelyrelated}. Note, the schemes built from the chain word $\mathcal{C}_{n,i,1}$ in Theorem \ref{thm_chainwordsnfr} are schemes for $V(M(A_i))$ since each $\mathcal{C}_{n,i,1}$ is $(i+1)$-limited. It follows that to show $V(M(B_i)) = V(M(A_i \cup \{a^ib^iab\}))$ is non-finitely related, we only need to show that the terms built from $\mathcal{C}_{n,i,1}$ form a scheme for $V(M(a^ib^iab))$ and do not come from a term. This is easily established, however, as $\mathbf{w} = a^ib^iab$ satisfies the requirements of Corollary \ref{cor_corchain}. Therefore each $V(M(B_i))$ is non-finitely related.
\end{proof}

Let $W$ be a set of words. For each $\mathbf{w}_i \in W$, let $X_{\mathbf{w}_i}$ be an alphabet such that $\lvert X_{\mathbf{w}_i} \rvert = \lvert c(\mathbf{w}_i) \rvert$ and  $X_{\mathbf{w}_i} \cap X_{\mathbf{w}_j} = \varnothing$ for all $\mathbf{w}_j \in W\setminus\{\mathbf{w}_i\}$. Let $\mathbf{u}_i$ be the word obtained from $\mathbf{w}_i$ by replacing (one-to-one) each $x \in c(\mathbf{w}_i)$ with a $y \in X_{\mathbf{w}_i}$. Let ~$\mathbf{u} = \prod_{\mathbf{w_i} \in W} \mathbf{u}_i$ so that $V(M(\mathbf{u})) = V(M(W))$ (see the proof of \cite[Theorem 4.4]{jackson2000finitely}). We may apply this trick to each $A_i$ in Corollary \ref{cor_alternatingfrnfr}, producing the word $\mathbf{A}_i$, and give the following sequence of subwords alternating between finitely and non-finitely related:
\[
\mathbf{A}_2 <  a^2b^2ab\mathbf{A}_2 < a^2b^2ab\mathbf{A}_2\mathbf{A}_3 < a^3b^3aba^2b^2ab\mathbf{A}_2\mathbf{A}_3 < \cdots
\]

\subsection{Schemes built from maelstrom words}

Like chain words, maelstrom words utilise a sort of commutativity to produce schemes and a pattern of isoterms to prevent those schemes coming from terms. Importantly, while these schemes will require that $x^py^px^qy^q$ is an isoterm (like the schemes built with chain words), they will not require chain words to commute. This provides the flexibility to study monoids where $x^{p+q}y^{p+q}$ is an isoterm. To our knowledge, maelstrom words have not gained any attention in the finite basis world, unlike their chain word cousins. 

For ease of use of these maelstrom words, we will only define them on an even number of variables. 

\begin{definition}
Given $n,p,q \in \mathbb{N}$, with $n$ even, the \emph{maelstrom word} on $n$ variables with variable exponents $p+q$ is 
given by
\[
    \mathcal{M}_{n,p,q} = x_1^p\cdot \Bigg(\prod_{k=1}^{\frac{n}{2}-1} x_{2k+1}^px_{2k}^p\Bigg)\cdot x_n^p \cdot \Bigg(\prod_{k=0}^{\frac{n}{2}-1}x_{n-(2k+1)}^qx_{n-2k}^q\Bigg).
\]
\end{definition}

One may also construct a maelstrom word in the following way: Let $X_{\text{even}} = \{x_2,x_4,\cdots\,x_n\}$ and $X_{\text{odd}} = \{x_1,x_3,\dots\,x_{n-1}\}$. Let $[X_{\text{even}}n] = x_2x_4\cdots x_n$ and $[nX_{\text{even}}] = x_nx_{n-2}\cdots x_2$. Define $[X_{\text{odd}}(n-1)]$ and $[(n-1)X_{\text{odd}}]$ analogously. Then $\mathcal{M}_{n,p,q}$, with powers $p,q$ removed, can be considered as an interlacing of $[X_{\text{even}}n][nX_{\text{even}}]$ and $[X_{\text{odd}}(n-1)][(n-1)X_{\text{odd}}]$:

\begin{figure}[h]
\centering
\begin{tikzpicture}
\node at (-4,2) {$x_1$};
\node at (-3.5,2) {$x_3$};
\node at (-2.5,2) {$x_5$};
\node at (-1,2) {$\cdots$};

\node at (-0.25,2) {$x_{n-1}$};
\node at (2.25,2) {$x_{n-3}$};
\node (v1) at (-3,1) {$x_2$};
\node (v3) at (-2,1) {$x_4$};
\node at (3.75,1) {$\cdots$};

\node (v7) at (0.5,1) {$x_{n-2}x_n$};
\node (v9) at (1.75,1) {$x_n$};
\node at (-1,1) {$\cdots$};
\node at (4.25,2) {$x_5$};
\node at (5.25,2) {$x_3$};
\node at (6.25,2) {$x_1$};
\node at (-1.5,2) {$x_7$};
\node at (1.25,2) {$x_{n-1}$};
\node at (3.75,2) {$\cdots$};
\node (v11) at (4.75,1) {$x_6$};
\node (v13) at (5.75,1) {$x_4$};
\node (v15) at (6.75,1) {$x_2$};
\node (v2) at (-3,2) {};
\node (v4) at (-2,2) {};

\node (v8) at (0.5,2) {};
\node (v10) at (1.75,2) {};
\node (v12) at (4.75,2) {};
\node (v14) at (5.75,2) {};
\node (v16) at (6.75,2) {};
\draw [->] (v1) edge (v2);
\draw [->] (v3) edge (v4);

\draw [->] (v7) edge (v8);
\draw [->] (v9) edge (v10);
\draw [->] (v11) edge (v12);
\draw [->] (v13) edge (v14);
\draw [->] (v15) edge (v16);
\node (v17) at (2.75,1) {$x_{n-2}$};
\node (v18) at (2.75,2) {};
\draw [->] (v17) edge (v18);
\end{tikzpicture}
\end{figure}

\begin{definition}\label{defn_maelstromodot}
Let $\mathcal{M}_{n,p,q}$ and $\mathbf{w}$ be maelstrom words. Then $\odot$ is defined as follows,
\[
    \mathcal{M}_{n,p,q} \odot \mathbf{w} := x_1^px_3^p\cdots x_{n-2}^px_n^p \cdot \mathbf{w} \cdot x_{n-1}^qx_n^q \cdots x_2^q.
\]
\end{definition}
\noindent Here, $\odot$ inserts $\mathbf{w}$ into the word $\mathcal{M}_{n,p,q}$ after the $p$-th occurrence of any variable $x_i$ but before its $(p+1)$-th occurrence. This ensures that for any variable $y \in c(\mathbf{w})$ we get 
\[
(\mathcal{M}_{n,p,q} \odot \mathbf{w})[x_i,y] = x_i^py^{\text{occ}(y,\mathbf{w})}x_i^q,
\]
avoiding the isoterms $x^py^px^qy^q$ and $x^{p+q}x^{p+q}$. 

To build the desired schemes, we first need to be able to ``slice" maelstrom words into several parts. Let $i,j,n \in \mathbb{N}$ with $n \geq \text{max}(i,j)$. If $i \leq j$ define $\mathcal{M}_{p,q}(i;j)$ by
\[
    \mathcal{M}_{p,q}(i;j) := \mathcal{M}_{n,p,q}[x_i,\dots,x_j].
    \]
If $j < i$ let $X' = \{x_1,\dots,x_j,x_i,\dots,x_n\}$ and define $\mathcal{M}_{p,q}(i;n;j)$ by
\[
\mathcal{M}_{p,q}(i;n;j) := \begin{cases}
    \mathcal{M}_{n,p,q}(x_i,\dots,x_n,x_1,\dots,x_{i-1})[X'] &\text{ if } i\text{ is odd,}\\
    \mathcal{M}_{n+2,p,q}(y,x_i,\dots,x_n,x_1,\dots,x_{i-1},z)[X'] &\text{ if } i \text{ is even}.
\end{cases}
\]

Our aim is to produce maelstrom words for subgraphs (with certain edge conditions) obtained by vertex deletion of the graph in Figure \ref{fig:maelstrom}. To fully capture this idea, we require some care for boundary cases. To this end, let $1$ denote the empty word in the free monoid $A^*$. For $i \in \underline{n}$ we define the following:
\begin{align*}
    \mathcal{M}_{p,q}(i+1;i) &:= 1, \\
    \mathcal{M}_{p,q}(i;n;0) &:= \mathcal{M}_{p,q}(i;n), \\
    \mathcal{M}_{p,q}(n+1;n;i) &:= \mathcal{M}_{p,q}(1;i),\\
    \mathcal{M}_{p,q}(n+1;n;0) &:= 1.
\end{align*}

The reader will notice the similarity between the following lemma and Lemma ~\ref{lemma_chainwordcommute}.
\begin{lemma}\label{lemma_maelstromcommute}
Let $\mathbf{S}$ be a monoid. Fix $p,q \in \mathbb{N}$ and suppose for any even $m,n \in \mathbb{N}$,
\[
\mathbf{S} \models \mathcal{M}_{n,p,q}\odot \overline{\mathcal{M}_{m,p,q}} \approx \overline{\mathcal{M}_{m,p,q}}\odot \mathcal{M}_{n,p,q},
\] 
where $\mathcal{M}_{n,p,q}$ and $\overline{\mathcal{M}_{m,p,q}}$ are maelstrom words over disjoint variable sets. Fix an even $n \in \mathbb{N}$ and let 
\[
\mathbf{t}_{ij} = \mathcal{M}_{p,q}(j+1;n;i-1)\odot\mathcal{M}_{p,q}(i+1,j-1),
\]
for $i < j \leq n$. Then $\mathbf{S}$ satisfies 
\begin{equation}\label{eqn_maelstromcommute}
    \mathbf{t}_{ij}[X_n\setminus\{x_j,x_k,x_l\}] \approx \mathbf{t}_{kl}[X_n\setminus\{x_i,x_j,x_l\}],
\end{equation}
for all $i,j,k,l \in \underline{n}$ with $i < j$ and $k < l$. 
\end{lemma}
\begin{proof}
Let $\mathbf{u}, \mathbf{v}$ be the words on the left and right hand side of (\ref{eqn_maelstromcommute}) respectively. It is easily seen that the content of both $\mathbf{u}$ and $\mathbf{v}$ coincide and equals $X_n \setminus \{x_i,x_j,x_k,x_l\}$. We will prove (\ref{eqn_maelstromcommute}) in the case for when $i < j < k < l$, but all cases are similar. A simple calculation shows
\begin{equation*}
\begin{aligned}
    \mathbf{u} = \mathcal{M}_{p,q}(j+1;k-1) &\odot \mathcal{M}_{p,q}(k+1;l-1)  \\
    & \odot \mathcal{M}_{p,q}(l+1;n;i-1) \odot\mathcal{M}_{p,q}(i+1,j-1).
\end{aligned}
\end{equation*}

Given that any two words of the form $\mathcal{M}_{p,q}(r;s)$ or $\mathcal{M}_{p,q}(r;n;s)$ are constructed by first rearranging and then deleting letters from $\mathcal{M}_{n,p,q}$ (or $\mathcal{M}_{n+2,p,q}$), the commutativity of any two such words under $\odot$ follows from $\mathbf{S} \models \mathcal{M}_{n,p,q}\odot \overline{\mathcal{M}}_{m,p,q} \approx \overline{\mathcal{M}}_{m,p,q}\odot \mathcal{M}_{n,p,q}$ provided they do not share any variables. Since $\mathbf{u}$ is made up of such words, and the variable sets for those words are pairwise disjoint, we obtain 
\begin{equation*}
\begin{aligned}
    \mathbf{u} \approx \mathcal{M}_{p,q}(l+1;n;i-1) &\odot \mathcal{M}_{p,q}(i+1,j-1) \\ &\odot \mathcal{M}_{p,q}(j+1;k-1) \odot \mathcal{M}_{p,q}(k+1;l-1) = \mathbf{v}. \\ 
\end{aligned}
\end{equation*}
\end{proof}

The commutativity of maelstrom words with respect to $\odot$ will have the same effect for the schemes we build as the commutativity of chain words had on the schemes in the previous section. 

\begin{figure}[H]
    \centering
    \begin{tikzpicture}

    \node[circle] (v1) at (-2,0.5) {$x_1$};

    \node [shape=circle] (v2) at (-0.5,2.5) {$x_2$};
    \node [shape=circle] (v3) at (1,0.5) {$x_3$};
    \node [shape=circle] (v4) at (2.5,2.5) {$x_4$};
    \node [shape=circle] (v5) at (3.25,1.56) {$\cdots$};
    \node [shape=circle] (v6) at (4,0.5) {$x_{n-1}$};
    \node [shape=circle] (v7) at (5.5,2.5) {$x_n$};
    \draw [->] (v1) edge (v2);

    \draw [->] (v3) edge (v4);

    \draw [->] (v6) edge (v7);
    \draw [->] (v3) edge (v2);
    \draw   [->](v5) edge (v4);
    \draw  [->] (v6) edge (v5);
    \draw [->,bend left=80] (v1) edge (v7);
    
    \end{tikzpicture}
    \caption{An impossible pattern for maelstrom words.}
    \label{fig:maelstrom}
\end{figure}
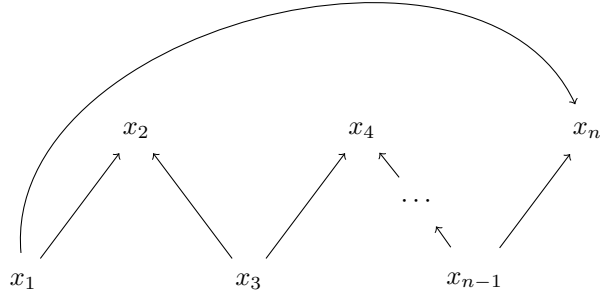

\begin{lemma}\label{lemma_maelstromforbiddenword}
Let $G = (V,E)$ be the directed graph in Figure \ref{fig:maelstrom} and let $n \geq 4$ be even. Fix $p,q \in \mathbb{N}$. Then there is no $n$-ary word $\mathbf{t}$ with the following properties:
\begin{enumerate}[label=(\roman*)]
    \item if $(x,y) \in E(G)$ then $\mathbf{t}[x,y] = x^py^px^qy^q$;
    \item if $(x,y) \not\in E(G)$ then $\mathbf{t}[x,y] \in \{x^py^{p+q}x^q, y^px^{p+q}y^q\}$.
\end{enumerate}
\end{lemma}
\begin{proof}
For the sake of contradiction, suppose there exists an $n$-ary word $\mathbf{t}$ satisfying the conditions (i) and (ii). Suppose $\mathbf{t}$ begins with $x_i$. Then $i$ must be odd since otherwise $\mathbf{t}[x_i,x_{i+1}] = x_{i+1}^px_{i}^px_{i+1}^qx_i^q$ contradicting the fact that $\mathbf{t}$ begins with $x_i$. Now, consider $\mathbf{t}[x_{i\ominus2},x_{i\ominus1},x_i,x_{i\oplus1},x_{i\oplus2}]$ where $\oplus$ and $\ominus$ are addition and subtraction modulo $n$ respectively. Since there are, pairwise, no edges between $x_{i\ominus2}, x_{i}$ and $x_{i\oplus2}$ in $G$ and $\mathbf{t}$ begins with $x_i$ either
\begin{equation}\label{eqn_maelstrom_forbidden1}
    \mathbf{t}[x_{i\ominus2},x_i,x_{i\oplus2}] = x_i^px_{i\ominus2}^px_{i\oplus2}^{p+q}x_{i\ominus2}^qx_i^q,
\end{equation}
or
\begin{equation}\label{eqn_maelstrom_forbidden2}
    \mathbf{t}[x_{i\ominus2},x_i,x_{i\oplus2}] = x_i^px_{i\oplus2}^px_{i\ominus2}^{p+q}x_{i\oplus2}^qx_i^q.
\end{equation}
If equation (\ref{eqn_maelstrom_forbidden1}) is true then since $\mathbf{t}[x_{i\oplus 1}, x_{i\oplus 2}] = x_{i\oplus 2}^px_{i\oplus 1}^px_{i\oplus 2}^qx_{i\oplus 1}^q$ and $\mathbf{t}[x_i,x_{i \oplus 1}] = x_i^px_{i\oplus 1}^px_i^qx_{i\oplus 1}^q$ we must have
\[
\mathbf{t}[x_{i\ominus2},x_i,x_{i\oplus1},x_{i\oplus2}] = x_i^px_{i\ominus2}^px_{i\oplus2}^{p}x_{i\oplus1}^px_{i\oplus2}^{q}x_{i\ominus2}^qx_i^qx_{i\oplus1}^q,
\]
which contradicts the fact that $(x_{i\ominus2},x_{i\oplus1}) \not\in E(G)$. Similarly, if equation (\ref{eqn_maelstrom_forbidden2}) is true then $\mathbf{t}[x_{i\ominus2},x_{i\ominus1},x_i,x_{i\oplus2}] = x_i^px_{i\oplus2}^px_{i\ominus2}^px_{i\ominus1}^px_{i\ominus2}^qx_{i\oplus2}^qx_i^qx_{i\ominus1}^q$ which contradicts the fact that $(x_{i\oplus 2},x_{i\ominus 1}) \not\in E(G)$. As every case has led to a contradiction, it follows that no such $\mathbf{t}$ can exist. 
\end{proof}

Given a semigroup $\mathbf{S}$, we will say that a set of words $\{\mathbf{u}_1,\dots,\mathbf{u}_n\}$
forms an \emph{island} for $\mathbf{S}$ if for all $i,j \in \underline{n}$ we have $\mathbf{S} \models \mathbf{u}_i \approx \mathbf{u}_j$ and for any word $\mathbf{v}$, $\mathbf{S}$ satisfies $\mathbf{u}_i \approx \mathbf{v}$ if and only if $\mathbf{v} = \mathbf{u}_k$ for some $k \in \underline{n}$.

\begin{theorem}\label{thm_maelstromewords}
Let $\mathbf{S}$ be a finite monoid. Suppose there exists $p,q \in \mathbb{N}$ such that $\mathbf{S}$ satisfies the following:
\begin{enumerate}[label=(\roman*)]
    \item for any even $m,n \in \mathbb{N}$, $\mathbf{S} \models \mathcal{M}_{n,p,q}\odot \overline{\mathcal{M}_{m,p,q}} \approx \overline{\mathcal{M}_{m,p,q}}\odot \mathcal{M}_{n,p,q}$;
    \item $\mathbf{S} \models \mathscr{A}_{\alpha,\beta}$ where $\alpha,\beta \in \mathbb{N}$ and $\alpha \leq 2p+2q$;
    \item $x^py^px^qy^q$ is an isoterm for $\mathbf{S}$;
    \item $x^py^{p+q}x^q \approx y^px^{p+q}y^q$ is an island identity for $\mathbf{S}$. 
\end{enumerate}
Then $\mathbf{S}$ is non-finitely related.
\end{theorem}
\begin{proof}
Let $n$ be even and $n > 4$. For any $1 \leq i < j \leq n$, let 
\[
\mathbf{t}_{ij} = x_j^{2p+2q}\cdot \mathcal{M}_{p,q}(j+1;n;i-1)\odot\mathcal{M}_{p,q}(i+1,j-1). 
\]
Let $\mathscr{F} = \{\mathbf{t}_{ij} \mid 1 \leq i < j \leq n\}$. We will show that $\mathscr{F}$ is an scheme for $V(\mathbf{S})$. Dependency follows immediately from the fact that $x_i \not\in c(\mathbf{t}_{ij})$. To show consistency, take any $i,j,k,l \in \underline{n}$ with $i < j$ and $k < l$. 
Condition (ii) ensures that $x_j$ and $x_l$ are strongly primitive in $\mathbf{t}_{ij}^{(kl)}$ and $\mathbf{t}_{kl}^{(ij)}$ with $\text{occ}(x_j,\mathbf{t}_{ij}^{(kl)}) = 2p+2q = \text{occ}(x_j,\mathbf{t}_{kl}^{(ij)})$ and $\text{occ}(x_l,\mathbf{t}_{ij}^{(kl)}) = 2p+2q = \text{occ}(x_l,\mathbf{t}_{kl}^{(ij)})$. Combining this fact and $(i)$ allows us to obtain $\mathbf{S} \models \mathbf{t}_{ij}^{(kl)}\approx \mathbf{t}_{kl}^{(ij)}$ from Lemma \ref{lemma_maelstromcommute}. Therefore $\mathscr{F}$ obeys the consistency condition hence $\mathscr{F}$ is an scheme for $V(\mathbf{S})$. 

Let $G$ denote the graph in Figure \ref{fig:maelstrom}. By the construction of the terms $\mathbf{t}_{ij}$ and ~(iii), if $\mathscr{F}$ did come from a term $\mathbf{t}$, then $\mathbf{t}[x,y] = x^py^qx^qy^q$ whenever $(x,y) \in E(G)$. Furthermore, if ~$(x,y) \not\in E(G)$, then the terms $\mathbf{t}_{ij}$ would also require that $\mathbf{t}[x,y] \in \{x^py^{p+q}x^q, y^px^{p+q}y^q\}$ by (iv). But by Lemma \ref{lemma_maelstromforbiddenword} no such term can exist. It follows that $\mathscr{F}$ does not come from a term and therefore $\mathbf{S}$ is non-finitely related. 
\end{proof}

\begin{eg}\label{eg_ababaabb}
    $M(abab,aabb)$ is non-finitely related. 
\end{eg}
\begin{proof}
    Conditions (ii)--(iv) of Theorem \ref{thm_maelstromewords} with $p=q=1$ hold trivially for $M(abab,aabb)$. Condition (i) can be obtained from the fact that any evaluation of the term $xy^2x$ is $0$ unless $x$ or $y$ is assigned $1$. 
\end{proof}

\subsection{Schemes built from crown words.}

Comparing the schemes built from chains and maelstrom words, the reader will notice that the basic recipe is the same. 
\begin{itemize}
    \item We first show that, given some sort of commutativity between \\ crown/maelstrom words, an appropriate substitution and then deletions of variables forms an identity (see Lemma \ref{lemma_chainwordcommute} and Lemma \ref{lemma_maelstromcommute}). 
    \item Under certain conditions, a collection of such identities forms a scheme (see Theorem \ref{thm_chainwordsnfr} and Theorem \ref{thm_maelstromewords}).
    \item The operations determined by the schemes must obey the cyclic edge relations seen in Figures \ref{fig_chainword} and \ref{fig:maelstrom}. This necessarily requires that these operations are not term functions (see Lemma \ref{lemma_chainnotterm} and Lemma \ref{lemma_maelstromforbiddenword}).
\end{itemize}

We implement the same procedure in this section to show that, for certain monoids, schemes built from crown words do not come from terms. 

\begin{definition}\label{defn_crownword}
Given $n,p,q \in \mathbb{N}$, the \emph{crown word} on $n$ variables is given by
\[
    \mathbf{\mathcal{R}}_{n,p,q} :=
    \begin{cases}
        x_1^p \cdot \Bigg(\displaystyle \prod_{k=1}^{\frac{n}{2}-1}x_{2k+1}^px_{2k}^{p+q}x_{2k-1}^q\Bigg) \cdot x_n^{p+q}x_{n-1}^q &\text{if } n \text{ is even}, \\
        
        x_1^p \cdot \Bigg(\displaystyle \prod_{k=1}^{\frac{n-1}{2}}x_{2k+1}^px_{2k}^{p+q}x_{2k-1}^q\Bigg) \cdot x_n^q &\text{if } n \text{ is odd}.
    \end{cases}
\]
\end{definition}

On the odd-indexed variables, $X_{\text{odd}}$, it is easy to see that $\mathcal{R}_{n,p,q}[X_{\text{odd}}] = \mathcal{C}_{n,p,q}$. To create $\mathcal{R}_{n,p,q}$ from $\mathcal{C}_{n,p,q}$, we insert the even-indexed variables within $\mathcal{C}_{n,p,q}$ as follows: 
\begin{itemize}
    \item if $x_j \in X_{\text{even}}\setminus\{x_n\}$ then insert $x_j^{p+q}$ immediately after the $p$-th occurrence of $x_{j+1}$ and immediately before the $(p+1)$-th occurrence of $x_{j-1}$;
    \item if $n$ is even, insert $x_n^{p+q}$ immediately after the $(p+q)$-th occurrence of $x_{n-2}$ and immediately before the $(p+1)$-th occurrence of $x_{n-1}$.
\end{itemize}

Let $i,j,n \in \mathbb{N}$ with $n \geq \text{max}(i,j)$. If $i \leq j$ then define $\mathcal{R}_{p}(i;j)$ by
\[
    \mathcal{R}_{p,q}(i;j) := \mathcal{R}_{n,p,q}[x_i,\dots,x_j].
\]
If $j < i$ let $X' = \{x_1,\dots,x_j,x_i,\dots,x_n\}$ and define $\mathcal{R}_{p,q}(i;n;j)$ by 
\[
    \mathcal{R}_{p,q}(i;n;j) := 
    \begin{cases}
        \mathcal{R}_{n,p,q}(x_i,\dots,x_n,x_1,\dots,x_{i-1})[X'] &\text{if } i\text{ is odd,} \\
        \mathcal{R}_{n+1,p,q}(y,x_i,\dots,x_n,x_1,\dots,x_{i-1})[X'] &\text{if } i\text{ is even.}
    \end{cases}
\]

As with maelstrom words, we are required to define some boundary cases. For $i \in \underline{n}$ we define the following:
\begin{align*}
    \mathcal{R}_{p,q}(i+1;i) &:= 1, \\
    \mathcal{R}_{p,q}(i;n;0) &:= \mathcal{R}_{p,q}(i;n), \\
    \mathcal{R}_{p,q}(n+1;n;i) &:= \mathcal{R}_{p,q}(1;i),\\
    \mathcal{R}_{p,q}(n+1;n;0) &:= 1.
\end{align*}

\begin{lemma}\label{lemma_crownwordscommute}
Let $\mathbf{S}$ be a monoid. Fix $p,q \in \mathbb{N}$ and suppose for any $m,n \in \mathbb{N}$
\[
    \mathbf{S} \models \mathcal{R}_{n,p,q}\overline{\mathcal{R}_{m,p,q}} \approx \overline{\mathcal{R}_{m,p,q}}\mathcal{R}_{n,p,q}, 
\]
where $\mathcal{R}_{n,p,q}$ and $\overline{\mathcal{R}_{m,p,q}}$ are crown words over disjoint variable sets. Fix $n \in \mathbb{N}$ and let
\[
    \mathbf{t}_{ij} = \mathcal{R}_{p,q}(j+1;n;i-1)\mathcal{R}_{p,q}(i+1,j-1) 
\] 
for any $i < j \leq n$. Then $\mathbf{S}$ satisfies
\begin{equation}\label{eqn_crowncommute}
    \mathbf{t}_{ij}[X_n\setminus\{x_k,x_l\}] \approx \mathbf{t}_{kl}[X_n\setminus\{x_i,x_j\}],
\end{equation}
for all $i,j,k,l \in \underline{n}$ with $i < j$ and $k < l$. 
\end{lemma}
\begin{proof}
    Let $\mathbf{u},\mathbf{v}$ be the words on the left and right hand side of (\ref{eqn_crowncommute}) respectively. Clearly $c(\mathbf{u}) = c(\mathbf{v})$. We will only prove the lemma for the case when $i < j < k < l$, but all cases can be obtained in a similar way. A simple calculation shows
\[
    \mathbf{u} = \mathcal{R}_{p,q}(j+1;k-1)\mathcal{R}_{p,q}(k+1;l-1)\mathcal{R}_{p,q}(l+1;n;i-1)\mathcal{R}_{p,q}(i+1,j-1).
\]
Each word of the form $\mathcal{R}_{p,q}$ that makes up $\mathbf{u}$ is obtained by first rearranging and then deleting letters from a crown word. Since crown words commute on $\mathbf{S}$, it follows that words of the form $\mathcal{R}_{p,q}$ also commute. We then get 
\[
    \mathbf{u} \approx \mathcal{R}_{p,q}(l+1;n;i-1)\mathcal{R}_{p,q}(i+1,j-1)\mathcal{R}_{p,q}(j+1;k-1)\mathcal{R}_{p,q}(k+1;l-1) = \mathbf{v},
\]
as required.
\end{proof}

In the previous sections, we saw that it is convenient to use a graph and an edge relation to show that certain operations cannot be term functions. For crown words, we re-use Figure \ref{fig:maelstrom}, now requiring that if $(x,y) \in E(G)$ then $t[x,y] = x^py^{p+q}x^q$. 

\begin{lemma}\label{lemma_crownforbiddenword}
Let $G = (V,E)$ be the directed graph in Figure \ref{fig:maelstrom} and let $n > 4$ be even. Fix $p,q \in \mathbb{N}$. Then there is no $n$-ary word $\mathbf{t}$ with the following properties:
\begin{enumerate}[label=(\roman*)]
    \item if $(x,y) \in E(G)$ then $\mathbf{t}[x,y] = x^py^{p+q}x^q$;
    \item if $(x,y) \not\in E(G)$ then $\mathbf{t}[x,y] \in \{x^py^px^qy^q,y^px^py^qx^q,x^{p+q}y^{p+q},y^{p+q}x^{p+q}\}$.
\end{enumerate}
\end{lemma}
\begin{proof}
    For the sake of contradiction suppose there exists an $n$-ary word $\mathbf{t}$ such that ~$\mathbf{t}$ satisfies (i) and (ii). Suppose $\mathbf{t}$ begins with $x_i$. Then using Figure \ref{fig:maelstrom}, $i$ must be odd since if $i$ were even then $\mathbf{t}[x_i,x_{i\oplus 1}] = x_{i\oplus1}^px_{i}^{p+q}x_{i\oplus1}^q$ contradicting the fact that $\mathbf{t}$ begins with $x_i$. Now, by (ii) one of the following is true:
    \begin{enumerate}[label=(\arabic*)]
        \item $\mathbf{t}[x_{i \ominus 1},x_{i \oplus 1}] = x_{i \ominus 1}^p x_{i \oplus 1}^px_{i \ominus 1}^q x_{i \oplus 1}^q$;
        \item $\mathbf{t}[x_{i \ominus 1},x_{i \oplus 1}] = x_{i \oplus 1}^p x_{i \ominus 1}^px_{i \oplus 1}^q x_{i \ominus 1}^q$;
        \item $\mathbf{t}[x_{i \ominus 1},x_{i \oplus 1}] = x_{i \ominus 1}^{p+q}x_{i \oplus 1}^{p+q}$;
        \item $\mathbf{t}[x_{i \ominus 1},x_{i \oplus 1}] =x_{i \oplus 1}^{p+q}x_{i \ominus 1}^{p+q}$.
    \end{enumerate}
    Suppose (1) is true. Then since $(x_{i \ominus 2},x_{i \ominus 1}) \in E(G)$ and $x_{i \ominus 2}$ and $x_{i \oplus 1}$ are not adjacent in $G$ we have 
    \[
        \mathbf{t}[x_{i \ominus 2},x_{i \ominus 1},x_{i \oplus 1}] = x_{i \ominus 2}^p x_{i \ominus 1}^p x_{i \oplus 1}^px_{i \ominus 1}^q x_{i \ominus 2}^q x_{i \oplus 1}^q.
    \]
    But as $x_i$ appears first and $\mathbf{t}[x_i,x_{i \oplus 1}] = x_i^px_{i \oplus 1}^{p+q}x_i^q$ we have $\mathbf{t}[x_i,x_{i \ominus 2}] = x_i^px_{i \ominus 2}^{p+q}x_i^q$ contradicting the fact that $(x_i,x_{i \ominus 2}) \not\in E(G)$. A similar contradiction can be found if we assume (2) is true.

    Now suppose that (3) is true. Note, since $\mathbf{t}$ begins with $x_i$ we have $\mathbf{t}[x_{i \ominus 2},x_i] \in \{x_i^px_{i \ominus 2}^px_i^qx_{i \ominus 2}^q, x_i^{p+q}x_{i \ominus 2}^{p+q}\}$ by (ii). If $\mathbf{t}[x_{i \ominus 2},x_i] = x_i^px_{i \ominus 2}^px_i^qx_{i \ominus 2}^q$ then 
    \[
        \mathbf{t}[x_{i \ominus 2},x_{i \ominus 1},x_i,x_{i \oplus 1}] = x_i^p x_{i \ominus 2}^p x_{i \ominus 1}^{p+q}x_{i \oplus 1}^{p+q} x_i^q x_{i \ominus 2}^q.
    \]
    But then $\mathbf{t}[x_{i \ominus 2},x_{i \oplus 1}] = x_{i \ominus 2}^px_{i \oplus 1}^{p+q} x_{i \ominus 2}^q$ contradicts the fact that $(x_{i \ominus 2},x_{i \oplus 1}) \not\in E(G)$. If $\mathbf{t}[x_{i \ominus 2},x_i] = x_i^{p+q}x_{i \ominus 2}^{p+q}$ then since $(x_{i},x_{i \ominus 1}),(x_{i},x_{i \oplus 1})  \in E(G)$ we have
    \[
         \mathbf{t}[x_{i \ominus 2},x_{i \ominus 1},x_i,x_{i \oplus 1}] = x_i^px_{i \ominus 1}^{p+q}x_{i \oplus 1}^{p+q}x_i^qx_{i \ominus 2}^{p+q}.
    \]
    But then $\mathbf{t}[x_{i \ominus 1},x_{i \ominus 2}] = x_{i \ominus 1}^{p+q}x_{i \ominus 2}^{p+q}$ contradicting the fact that $(x_{i \ominus 2},x_{i \ominus 1}) \in E(G)$. A similar contradiction can be found for (4). 

    Since cases (1) through (4) are exhaustive by (ii), and each results in a contradiction, no word $\mathbf{t}$ satisfying the requirements of the lemma's statement can exist. 
\end{proof}

The proof of the following theorem is similar to Theorem \ref{thm_chainwordsnfr} and \ref{thm_maelstromewords}.

\begin{theorem}\label{thm_crownwords}
Let $\mathbf{S}$ be a finite monoid. Suppose there exists $p,q \in \mathbb{N}$ such that $\mathbf{S}$ satisfies the following:
\begin{enumerate}[label=(\roman*)]
    \item for any $m,n \in \mathbb{N}$, $\mathbf{S} \models \mathcal{R}_{n,p,q}\overline{\mathcal{R}_{m,p,q}} \approx \overline{\mathcal{R}_{m,p,q}}\mathcal{R}_{n,p,q}$;
    \item $\mathbf{S} \models \mathscr{A}_{\alpha,\beta}$, where $\alpha,\beta \in \mathbb{N}$ and $\alpha \leq 2p+2q$;
    \item $x^py^{p+q}x^q$ is an isoterm for $\mathbf{S}$;
    \item $\{x^py^px^qy^q, y^px^py^qx^q, x^{p+q}y^{p+q}, y^{p+q}x^{p+q}\}$ is an island for $\mathbf{S}$.
\end{enumerate}
Then $\mathbf{S}$ is non-finitely related.
\end{theorem}
\begin{proof}
Let $n$ be even and $n > 4$. For any $1 \leq i < j \leq n$, let 
\[
\mathbf{t}_{ij} = x_j^{2p+2q}\cdot \mathcal{R}_{p,q}(j+1;n;i-1)\mathcal{R}_{p,q}(i+1,j-1). 
\]
Let $\mathscr{F} = \{\mathbf{t}_{ij} \mid 1 \leq i < j \leq n\}$. We will show that $\mathscr{F}$ is a scheme for $V(\mathbf{S})$. Dependency follows immediately from the fact that $x_i \not\in c(\mathbf{t}_{ij})$. To show consistency, take any $i,j,k,l \in \underline{n}$ with $i < j$ and $k < l$. 
Condition (ii) ensures that $x_j$ and $x_l$ are strongly primitive in $\mathbf{t}_{ij}^{(kl)}$ and $\mathbf{t}_{kl}^{(ij)}$ with $\text{occ}(x_j,\mathbf{t}_{ij}^{(kl)}) = 2p+2q = \text{occ}(x_j,\mathbf{t}_{kl}^{(ij)})$ and $\text{occ}(x_l,\mathbf{t}_{ij}^{(kl)}) = 2p+2q = \text{occ}(x_l,\mathbf{t}_{kl}^{(ij)})$. Combining this fact and (i) allows us to obtain $\mathbf{S} \models \mathbf{t}_{ij}^{(kl)}\approx \mathbf{t}_{kl}^{(ij)}$ from Lemma \ref{lemma_crownwordscommute}. Therefore $\mathscr{F}$ obeys the consistency condition hence $\mathscr{F}$ is a scheme for $V(\mathbf{S})$. 

Let $G$ denote the graph in Figure \ref{fig:maelstrom}. By the construction of the terms $\mathbf{t}_{ij}$ and ~(iii), if $\mathscr{F}$ did come from a term $\mathbf{t}$, then $\mathbf{t}[x,y] = x^py^{p+q}x^q$ whenever $(x,y) \in E(G)$. Furthermore, if ~$(x,y) \not\in E(G)$, then by (iv) the terms $\mathbf{t}_{ij}$ would also require that $\mathbf{t}[x,y] \in \{x^py^px^qy^q,y^px^py^qx^q,x^{p+q}y^{p+q},y^{p+q}x^{p+q}\}$. But by Lemma \ref{lemma_crownforbiddenword} no such term can exist. It follows that $\mathscr{F}$ does not come from a term and therefore $\mathbf{S}$ is non-finitely related. 
\end{proof}

\begin{eg}\label{eg_abba}
    $M(abba)$ is non-finitely related.
\end{eg}
\begin{proof}
    Conditions (iii), and (iv) of Theorem \ref{thm_crownwords} clearly hold for $M(abba)$ with $p = q = 1$. As $M(abba) \models \mathscr{A}_{3,1}$ and $3 < 2p+2q = 4$, the monoid $M(abba)$ satisfies condition (ii) of Theorem \ref{thm_chainwordsnfr}. Condition (i) can be obtained by noticing that any evaluation of the term $x^2y^2$ is 0 unless $x$ or $y$ is assigned 1. 
\end{proof}
We can now establish the existence of two non-finitely related semigroups whose direct product is finitely related. 

\begin{corollary}\label{cor_existencetwonfrisfr}
There exists two non-finitely related semigroups whose direct product is finitely related.
\end{corollary}
\begin{proof}
Let $\mathbf{S} = M(abba)$ and $\mathbf{T} = M(abab,aabb)$. Then $\mathbf{S}$ and $\mathbf{T}$ are non-finitely related by Example \ref{eg_abba} and Example \ref{eg_ababaabb} respectively. But \[
V(\mathbf{S} \times \mathbf{T}) = V(M(\{a,b\}_2),\] which is finitely related by Theorem \ref{thm_makfinitelyrelated}. Therefore $\mathbf{S} \times \mathbf{T}$ is finitely related by Theorem \ref{thm_frvarietalproperty}.
\end{proof}

To the author's knowledge, Corollary \ref{cor_existencetwonfrisfr} appears to be the first example of two non-finitely related algebras whose direct product is finitely related.

Steindl has shown that finite relatedness is not preserved under homomorphic images, subsemigroups or Rees quotients \cite{steindl2022}. While we have given an example of two non-finitely related semigroups whose direct product is finitely related, it still remains unknown whether finite relatedness is preserved by direct products in the semigroup context. 

\begin{problem}
Do there exist finitely related semigroups $\mathbf{S}$ and $\mathbf{T}$ such that $\mathbf{S} \times \mathbf{T}$ is not finitely related?
\end{problem}

\section{A short note on the addition of an identity to non-finitely related semigroups}

The existence of a non-finitely related nilpotent monoid satisfies the question of whether finite relatedness is preserved by the addition of an identity element (see Steindl \cite{steindl2022}). We may also ask whether we can obtain a finitely related monoid from a non-finitely related semigroup. In the world of the finite basis problem, a non-finitely based semigroup $\mathbf{S}$ is \emph{conformable} if $\mathbf{S}^1$ is finitely based. Lee was the first to show that finite conformable semigroups exist \cite{lee2013finitely}. Lee did this by showing that, under certain circumstances, the direct product of a suitable monoid with a suitable nilpotent semigroup would suffice. Lee's paper on conformable semigroups translates, almost trivially, from the finite basis world to the finite relatedness world. 

The only if-part of the next lemma follows from \cite[Lemma 3.11]{Davey2011}.

\begin{lemma}\label{lemma_directproductnilpotent}
Let $\mathbf{S}$ and $\mathbf{N}$ be a semigroup and nilpotent semigroup respectively. Then $\mathbf{S}$ is finitely related if and only if $\mathbf{S} \times \mathbf{N}$ is finitely related.
\end{lemma}
\begin{proof}
    The only if-part follows from \cite[Lemma 3.11]{Davey2011}. We prove the if-statement. 

    Suppose $\mathbf{S} \times \mathbf{N}$ is finitely related with term degree $k$. To show that $\mathbf{S}$ is finitely related, it suffices to show that $\mathbf{S}^0$ is finitely related by \cite[Theorem 4.1]{Mayr2013}. As $\mathbf{S} \times \mathbf{N}$ is finitely related, $(\mathbf{S} \times \mathbf{N})^0$ is finitely related by \cite[Theorem 4.1]{Mayr2013}. It follows from Theorem \ref{thm_frvarietalproperty} that $\mathbf{S}^0 \times \mathbf{N}$ is finitely related since $V((\mathbf{S} \times \mathbf{N})^0) = V(\mathbf{S}^0 \times \mathbf{N})$. 
    
    Let $d \in \mathbb{N}$ be the term degree of $\mathbf{S}^0 \times \mathbf{N}$ and $\mathscr{F} = \{\mathbf{t}_{ij}\mid 1 \leq i < j \leq n\}$ be an $n$-ary scheme for $\mathbf{S}^0$ that depends on all of its variables with $n > \text{max}(d,\lvert S^0 \rvert + 1)$. Then by Lemma \ref{lemma_containssemilattice} we have $\lvert c(\mathbf{t}_{ij}) \rvert = n - 1 \geq d \geq \lvert S^0 \times N \lvert > \lvert N \rvert$. But $\mathbf{N}$ is ~$\lvert N \rvert$-nilpotent by \cite[Lemma 3.3]{Davey2011}, hence $\mathscr{F}$ is a scheme for $\mathbf{N}$ and thus a scheme for $\mathbf{S}^0 \times \mathbf{N}$. It follows that $\mathscr{F}$ comes from a term ~$\mathbf{t}$ for $\mathbf{S}^0 \times \mathbf{N}$ since $\mathbf{S}^0 \times \mathbf{N}$ has term degree $d < n$. Therefore $\mathscr{F}$ comes from the term $\mathbf{t}$ for $\mathbf{S}^0$ and so $\mathbf{S}^0$ has term degree at most $n$. 
\end{proof}

We will now prove two theorems corresponding to Theorems 4 and 5 in Lee's paper on conformable semigroups \cite{lee2013finitely}. Following Lemma \ref{lemma_directproductnilpotent}, this essentially amounts to replacing the words ``finitely based" and ``non-finitely based" in Lee's paper with ``finitely related" and ``non-finitely related" respectively. 

\begin{theorem}\label{thm_conformablesemigroup1}
Suppose that $\mathbf{S}$ and $\mathbf{N}$ are semigroups such that 
\begin{enumerate}[label=(\roman*)]
    \item $\mathbf{S}^1$ is non-finitely related;
    \item $\mathbf{N}$ is nilpotent;
    \item $\mathbf{S}^1 \times \mathbf{N}^1$ is finitely related.
\end{enumerate}
Then $\mathbf{S}^1 \times \mathbf{N}$ is non-finitely related but $(\mathbf{S}^1 \times \mathbf{N})^1$ is finitely related.
\end{theorem}
\begin{proof}
The fact that $\mathbf{S}^1 \times \mathbf{N}$ is non-finitely related follows from (i), (ii) and Lemma \ref{lemma_directproductnilpotent}. Following Lee's argument, $V((\mathbf{S}^1 \times \mathbf{N})^1) = V(\mathbf{S}^1 \times \mathbf{N}^1)$ and hence $(\mathbf{S}^1 \times \mathbf{N})^1$ is finitely related by Theorem \ref{thm_frvarietalproperty}. 
\end{proof}

\begin{theorem}\label{thm_conformablesemigroup2}
Suppose that $\mathbf{S}$ and $\mathbf{N}$ are semigroups such that 
\begin{enumerate}[label=(\roman*)]
    \item $\mathbf{S}^1$ is non-finitely related;
    \item $\mathbf{N}$ is nilpotent;
    \item $\mathbf{N}^1$ is finitely related;
    \item $V(\mathbf{S}^1) \subseteq V(\mathbf{N}^1)$.
\end{enumerate}
Then $\mathbf{S}^1 \times \mathbf{N}$ is non-finitely related but $(\mathbf{S}^1 \times \mathbf{N})^1$ is finitely related.
\end{theorem}
\begin{proof}
From (i), (ii), and Lemma \ref{lemma_directproductnilpotent} we have $\mathbf{S}^1 \times \mathbf{N}$ is non-finitely related. From (iv) and Theorem \ref{thm_conformablesemigroup1} we get $V((\mathbf{S}^1 \times \mathbf{N})^1) = V(\mathbf{S}^1 \times \mathbf{N}^1) = V(\mathbf{N}^1)$ which is finitely related by (iii). Therefore $(\mathbf{S}^1 \times \mathbf{N})^1$ is finitely related by Theorem \ref{thm_frvarietalproperty}.
\end{proof}

\begin{corollary}\label{cor_conformal}
There exists a non-finitely related semigroup $\mathbf{T}$ such that $\mathbf{T}^1$ is finitely related.
\end{corollary}
\begin{proof}
Let $\mathbf{S} = S(abab)$ and $\mathbf{N} = S(abab,abba,aabb)$. Then $\mathbf{S}$ and $\mathbf{N}$ satisfy the conditions of Theorem \ref{thm_conformablesemigroup2} by Corollary \ref{cor_corchain} and Theorem \ref{thm_makfinitelyrelated}. Therefore $\mathbf{T} = \mathbf{S}^1 \times \mathbf{N}$ is non-finitely related, but $\mathbf{T}^1$ is finitely related. 
\end{proof}

\section{A finitely related non-finitely based semigroup.}

This paper thus far has presented results that are suggestive of an unexpected transfer of a positive and negative properties from the finite basis world to the finite relatedness of nilpotent monoids built from words. Interestingly, this (perceived) connection extends beyond the specific semigroups we have studied in this paper:
\begin{itemize}
    \item All groups are finitely based and finitely related \cite{Aichinger2014}.
    \item Commutative semigroups, nilpotent semigroups and bands are both finitely based and finitely related \cite{Davey2011,Dolinka2018}.
    \item Clifford semigroups are finitely based and finitely related \cite{Mayr2013}.
    \item The $6$-element Brandt monoid $\mathbf{B}_2^1$ is non-finitely based \cite{Perkins1969} and non-finitely related \cite{Mayr2013}.
    \item The specific nilpotent monoid Steindl showed was non-finitely related is non-finitely based by applying a result of Jackson and Sapir \cite[Lemma 4.3]{jackson2001finite}.
\end{itemize}

To the author's knowledge, each semigroup that has been shown to be finitely related has also been finitely based. On the other hand the few examples of non-finitely related semigroups have been non-finitely based. There is no obvious link between the two concepts, and we do not mean to imply that such a link exists, but the coincidence up to now is interesting nonetheless. However, this coincidence does not extend to algebras beyond semigroups. For instance, one may consider a pointed group: an algebra $\langle G, \cdot, g \rangle$ where $\langle G, \cdot \rangle$ is a group and $g$ is a fixed element of $G$ considered as a nullary operation. Certainly every finite pointed group has few subpowers, so is finitely related \cite{Aichinger2014}, but there exists a finite non-finitely based pointed group \cite{bryant1982laws}.

We finish this paper by providing a semigroup example: the non-finitely based monoid $M(asabtb)$ \cite[Lemma 5.5]{jackson2005finiteness} is finitely related. 

Given a word $\mathbf{w}$, the linear variables of $\mathbf{w}$ are the variables $x \in c(\mathbf{w})$ such that $\text{occ}(x,\mathbf{w}) = 1$. The proof of the following lemma can be found in the proof of \cite[Theorem 2.3]{steindl2022}.
\begin{lemma}[{\cite{steindl2022}}]\label{lemma_linearvariables}
    Let $\mathbf{S}$ be a non-commutative nilpotent monoid. Let ~$\mathscr{F} = \{\mathbf{t}_{ij} \mid 1 \leq i < j \leq n\}$ be a scheme for $\mathbf{S}$ with ~$n >\lvert S \rvert + 1$ inducing an operation $f:S^n \to S$. If $Y \subset X_n$ are the linear variables of $\mathscr{F}$ then $f[Y]$ is induced by the isoterm
    \[
    x_{i_1}x_{i_2}\cdots x_{i_m},
    \]
    where $Y = \{x_{i_1},\dots,x_{i_m}\}$ and $x_{i_p} \neq x_{i_q}$ for any $p \neq q$.
\end{lemma}

Given a scheme $\mathscr{F}$, we will denote the set of linear letters of $\mathscr{F}$ by $\text{Lin}(\mathscr{F})$. 

Recall that given a word $\mathbf{u}$ and a letter $x \in c(\mathbf{u})$, the $p$-th occurrence of $x$ in $\mathbf{u}$ is denoted by ${}_px$. Let $\text{OccSet}(\mathbf{u}) = \{{}_px \mid x \in c(\mathbf{u}), 1 \leq p \leq \text{occ}(x,\mathbf{u})\}$. Given a semigroup $\mathbf{S}$ and an identity $\mathbf{u} \approx \mathbf{v}$, a set $Y \subset \text{OccSet}(\mathbf{u})$ is \emph{unstable} in $\mathbf{u}\approx \mathbf{v}$ if for any two elements ${}_px,{}_qy \in Y$, the order in which ${}_px$ and ${}_qy$ occur in $\mathbf{v}$ differs from the order they occur in $\mathbf{u}$. We will say $Y$ is unstable in $\mathbf{u}$ if there exists a word $\mathbf{v}$ such that $\mathbf{S} \models \mathbf{u} \approx \mathbf{v}$ and $Y$ is unstable in $\mathbf{u} \approx \mathbf{v}$. A set $Y \subset \text{OccSet}(\mathbf{u})$ is stable if it is not unstable. 

\begin{prop}\label{prop_asabtbstable}
Let $\mathbf{S} = M(asabtb)$ and $\mathbf{u}$ be a word over an alphabet $A$ with $\mathbf{u}[Lin(\mathbf{u})] = t_1t_2\cdots t_m$. Suppose that $\mathbf{u}[x,t_{i-1},t_{i}] = t_{i-1}xt_ix$ and $\mathbf{u}[y,t_{i-1},t_{i}] = t_{i-1}yt_iy$. Then $\{{}_1x,{}_1y\}$ is stable in $\mathbf{u}$ if and only if there exists a variable $z \in c(\mathbf{u})$ such that $\mathbf{u}[z,t_{i-1},t_{i}] = zt_{i-1}zt_i$ and $xzy$ or $yzx$ is a subword of $\mathbf{u}[x,y,z,t_i,t_{i-1}]$. 

Similarly, if $\mathbf{u}[x,t_{i-1},\mathbf{t}_{i}] = xt_{i-1}xt_i$ and $\mathbf{u}[y,t_{i-1},t_{i}] = yt_{i-1}yt_i$ then $\{{}_2x,{}_2y\}$ is stable in $\mathbf{u}$ if and only if there exists $z \in c(\mathbf{u})$ such that $\mathbf{u}[z,t_{i-1},t_{i}] = t_{i-1}zt_iz$ and $xzy$ or $yzx$ is a subword of $\mathbf{u}[x,y,z,t_i,t_{i-1}]$. 
\end{prop}
\begin{proof}
    The backwards implication follows immediately from the fact that $zt_{i-1}xzt_ix$ and $zt_{i-1}zxt_ix$ are isoterms for $\mathbf{S}$. 

    To establish the forwards implication, we show the contrapositive is true. Without loss of generality, suppose ${}_1x$ occurs before ${}_1y$ in $\mathbf{u}$. By the assumption of the contrapositive, if an instance of $z \in C(\mathbf{u})$ occurs between ${}_1x$ and ${}_1y$ then $z$ is primitive in $\mathbf{u}$ or ${}_1z$ occurs between ${}_1x$ and ${}_1y$ and ${}_2z$ occurs after $t_i$. 
    
    It is enough to show that if ${}_1z$ occurs immediately after ${}_1x$ then $\{{}_1x,{}_1z\}$ is unstable in $\mathbf{u}$. Clearly $\{{}_1x,{}_1z\}$ is unstable if $z$ is primitive. Suppose $z$ is not primitive in $\mathbf{u}$. Let $\mathbf{v}$ be the word obtained from $\mathbf{u}$ by swapping ${}_1x{}_1z$ with ${}_1z{}_1x$. Consider any substitution $\Theta: A \to S$. Since neither $x$ nor $z$ are primitive or linear, $x$ and $z$ are 2-occurring. It follows that if neither $\Theta(x) = 1$ nor $\Theta(z) = 1$ then $\Theta(\mathbf{u}) = 0 = \Theta(\mathbf{v})$. Clearly if $\Theta(x) = 1$ or $\Theta(z) = 1$ then $\Theta(\mathbf{u}) = \Theta(\mathbf{v})$. Therefore $\mathbf{S} \models \mathbf{u} \approx \mathbf{v}$ and hence $\{{}_1x,{}_1z\}$ is unstable in $\mathbf{u}$. 

    The case for the stability of $\{{}_2x,{}_2y\}$ holds by a symmetric argument.
\end{proof}

Following Proposition \ref{prop_asabtbstable}, we will construct a normal form for any word $\mathbf{v}$ such that $M(asabtb) \models \mathbf{u} \approx \mathbf{v}$. To this end, assume $\mathbf{u}[\text{Lin}(\mathbf{u})] = t_1\cdots t_m$ and fix $i$ with $2 \leq i \leq m$. Let 
\[
{}_1X = \{x \in A \mid \mathbf{u}[t_{i-1},t_i,x] = t_{i-1}xt_ix\},
\]
and 
\[
{}_2X = \{x \in A \mid \mathbf{u}[t_{i-1},t_i,x] = xt_{i-1}xt_i\}.
\]
For each $x \in {}_1X$ let
\[
\mathcal{A}_x = \{y \in {}_1X \mid \{{}_1x,{}_1y\} \text{ is unstable in } \mathbf{u}\} \cup \{x\},
\]
and for each $x \in {}_2X$
\[
\mathcal{B}_x = \{y \in {}_2X \mid \{{}_2x,{}_2y\} \text{ is unstable in } \mathbf{u}\} \cup \{x\}.
\]
Using these sets, we obtain the following proposition.
\begin{prop}\label{prop_asabtbpartition}
    The sets $\{\mathcal{A}_x \mid x \in {}_1X\}$ and $\{\mathcal{B}_x \mid x \in {}_2X\}$ partition ${}_1X$ and ${}_2X$ respectively. 
\end{prop}
\begin{proof}
    We prove the case for ${}_1X$ only as the case for ${}_2X$ follows from a symmetric argument.

    Note, it suffices to show that for all $y \in \mathcal{A}_x$ we have $\mathcal{A}_y = \mathcal{A}_x$. Suppose $y,z \in \mathcal{A}_x$. By Proposition \ref{prop_asabtbstable}, for all $x' \in {}_2X$ neither $xx'y$ nor $yx'x$ are subwords of $\mathbf{u}[t_{i-1},t_i,x',x,y]$. Similarly, for all $x' \in {}_2X$ neither $xx'z$ nor $zx'x$ are subwords of $\mathbf{u}[t_{i-1},t_i,x',x,z]$. It follows that for all $x' \in {}_2X$ neither $yx'z$ nor $zx'y$ are subwords of $\mathbf{u}[t_{i-1},t_i,x',y,z]$ and hence $\{{}_1y,{}_1z\}$ is unstable in $\mathbf{u}$. Therefore $\mathcal{A}_x \subseteq \mathcal{A}_y$. The reverse inclusion is similar. The proposition is proved. 
\end{proof}
\begin{construction}\label{construction_asabtb}
    \normalfont By Proposition \ref{prop_asabtbstable} either the first occurrence of every variable in $\mathcal{A}_x$ occurs before the second occurrence of every variable in $\mathcal{B}_x$ or vice-versa. This, with Proposition \ref{prop_asabtbpartition}, gives a natural linear order on the blocks $\{\mathcal{A}_x \mid x \in {}_1X\}$ and $\{\mathcal{B}_x \mid x \in {}_2X\}$ of ${}_1X \cup {}_2X$. Without loss of generality, suppose this linear order is given by
\[
    \mathcal{A}_{x_1} < \mathcal{B}_{x_2} < \mathcal{A}_{x_3} < \mathcal{B}_{x_4} < \cdots < \mathcal{A}_{x_{n-1}} < \mathcal{B}_{x_n}.
\]

For each $\mathcal{A}_x$ (resp.\ $\mathcal{B}_x$), choose any word $\mathbf{v}_{x}$ such that $c(\mathbf{v}_{x}) = \mathcal{A}_x$ (resp.\  $c(\mathbf{v}_{x}) = {B}_x$) and $\text{occ}(y,\mathbf{v}_{x}) = 1$ for all $y \in \mathcal{A}_x$ (resp.\ $y \in \mathcal{B}_x$). Set $\mathbf{w}_{i} = \mathbf{v}_{x_1}\mathbf{v}_{x_2}\cdots \mathbf{v}_{x_n}$. 

Finally, let ${}_1X_0 = \{x \in A \mid \mathbf{u}[t_1,x] = xt_1x\}$ and ${}_2X_m = \{x \in A \mid \mathbf{u}[t_m,x] = xt_mx\}$. Produce the word $\mathbf{w}_0$ by choosing any word with $c(\mathbf{w}_0) = {}_1X_0$ and $\text{occ}(x,\mathbf{w}_{0}) = 1$. Produce $\mathbf{w}_{m+1}$ analogously. 
\end{construction}

From Proposition \ref{prop_asabtbstable} and Construction \ref{construction_asabtb} we gain the following corollary.

\begin{corollary}
    \label{cor_asabtb}
    Let $\mathbf{S} = M(asabtb)$ and $\mathbf{u}$ be a word with $\mathbf{u}[Lin(\mathbf{u})] = t_1t_2\cdots t_m$. Produce the words $\mathbf{w}_0,\mathbf{w}_1,\cdots,\mathbf{w}_{m+1}$ as in Construction \ref{construction_asabtb}. Let 
    \[
    \mathbf{v} = \mathbf{w}_0t_1\mathbf{w}_1t_2\cdots \mathbf{w}_m t_m \mathbf{w}_{m+1} \cdot \prod_{x \in \text{Prim}(\mathbf{u})} x^2.
    \]
    Then $\mathbf{S} \models \mathbf{u} \approx \mathbf{v}$. 
\end{corollary}

\begin{lemma}\label{lemma_asabtb}
    Let $\mathbf{S} = M(asabtb)$ and $\mathbf{u}, \mathbf{v}$ be words over $A$ with $c(\mathbf{u}) = c(\mathbf{v})$. Then $\mathbf{S} \models \mathbf{u} \approx \mathbf{v}$ if and only if $\mathbf{S} \models \mathbf{u}[X] \approx \mathbf{v}[X]$ for every $X \subseteq A$ with $\lvert X \rvert \leq 5$.
\end{lemma}
\begin{proof}
    The only if-direction is trivial. To show the if-statement suppose $\mathbf{S} \models \mathbf{u}[X] \approx \mathbf{v}[X]$ for every $X \subseteq A$ with $\lvert X \rvert \leq 5$. Then any letter $t_i$ that is linear in $\mathbf{u}$ must be linear in $\mathbf{v}$. Moreover, $\mathbf{u}[\text{Lin}(\mathbf{u})] = \mathbf{v}[\text{Lin}(\mathbf{v})]$ can be ascertained from the fact that $\mathbf{u}[t_i,t_j] \approx \mathbf{v}[t_i,t_j]$ for any linear letters $t_i$ and $t_j$. 
    
    Notice that a letter $x$ is primitive in a word for $\mathbf{S}$ if and only if it occurs more than twice or there are no linear letters between two occurrences of $x$. We can then establish $\text{Prim}(\mathbf{u}) = \text{Prim}(\mathbf{v})$ by considering $\mathbf{u}[t_i,x]$ and $\mathbf{v}[t_i,x]$ for every letter $x$ and every linear letter $t_i$. 
    
    Now, since $\mathbf{S} \models \mathbf{u}[X] \approx \mathbf{v}[X]$ for any set of letters $X$ with $\lvert X \rvert = 5$, $\{{}_px,{}_py\}$ is stable in $\mathbf{u}$ if and only if $\{{}_px,{}_py\}$ is stable in $\mathbf{v}$ by Proposition \ref{prop_asabtbstable} for $p \in \{1,2\}$.

    Consider Construction \ref{construction_asabtb} and the normal form in Corollary \ref{cor_asabtb}. Not only do the sets $_{1}X$ and $_{2}X$ have to coincide for $\mathbf{u}$ and $\mathbf{v}$, by the above facts, but since the stability of $\{{}_px,{}_py\}$ is consistent across $\mathbf{u}$ and $\mathbf{v}$, the sets $\mathcal{A}_x$ and $\mathcal{B}_x$ have to coincide for $\mathbf{u}$ and $\mathbf{v}$. Moreover, the natural order on ${}_1X \cup {}_2X$ in $\mathbf{u}$ must be the same order on ${}_1X \cup {}_2X$ in $\mathbf{v}$. Therefore we have some word $\mathbf{w}_0t_1\mathbf{w}_1t_2\cdots \mathbf{w}_m t_m \mathbf{w}_{m+1}$ such that
    \[
    \mathbf{u} \approx \mathbf{w}_0t_1\mathbf{w}_1t_2\cdots \mathbf{w}_m t_m \mathbf{w}_{m+1}\cdot\prod_{x \in \text{Prim}(\mathbf{u})} x^2 \approx \mathbf{v} 
    \]
    by Corollary \ref{cor_asabtb}. Thus the lemma is proved. 
\end{proof}
Given an $n$-ary scheme $\mathscr{F}$ for a semigroup $\mathbf{S}$, let $\text{Prim}(\mathscr{F}) = \{x_i \in X_n \mid \forall k,l \in \underline{n}\setminus\{i\} \  x_i \in \text{Prim}(\mathbf{t}_{kl})\}$.

\begin{theorem}\label{thm_asabtbfr}
    The nilpotent monoid $M(asabtb)$ is finitely related.
\end{theorem}
\begin{proof}
        Let $\mathbf{S} = M(asabtb)$ and $\mathscr{F} = \{\mathbf{t}_{ij} \mid 1 \leq i < j \leq n\}$ be a scheme for $\mathbf{S}$ with ~$n >\lvert S \rvert + 1$ inducing the operation $f:S^n \to S$. Without loss of generality, assume $\{1,\dots,m\}$ indexes $\text{Lin}(\mathscr{F})$ and $f[x_1,\dots,x_m]$ is induced by ~$x_1x_2\cdots x_m$ (see Lemma \ref{lemma_linearvariables}). 
        
        For any subset $Y \subset X_n$ with $\lvert Y \rvert \leq 21$ the operation $f[Y]$ is induced by a term since $n > \lvert S \rvert + 1 \geq 22$. 
        We may then construct $\mathbf{v}$ as in Corollary \ref{cor_asabtb} using the terms $f[Y]$ for $\lvert Y \rvert \leq 5$.
        Perform this construction and let 
        \[
        \mathbf{v} = \mathbf{w}_0x_1\mathbf{w}_1x_2\cdots \mathbf{w}_m\mathbf{x}_m\mathbf{w}_{m+1} \cdot \prod_{x \in \text{Prim}(\mathscr{F})}x^2.
        \]
        We claim that $\mathscr{F}$ comes from the term $\mathbf{v}$. Take any subset $Y \subset X_n$ with $\lvert Y \rvert = 5$ and let $\mathbf{s}$ be the term inducing $f[Y]$. It is easy to verify $\mathbf{S} \models \mathbf{s} \approx \mathbf{v}[Y]$ by considering the following facts:
        \begin{itemize}
            \item $\mathbf{s}[x_{m_1},x_{m_2}] = \mathbf{v}[x_{m_1},x_{m_2}]$ for all linear variables $x_{m_1},x_{m_2}$;
            \item by the construction of $\mathbf{v}$, $\mathbf{s}[x_{m_1},x_{m_2},x_i,x_j,x_k] \approx \mathbf{v}[x_{m_1},x_{m_2},x_i,x_j,x_k]$ where $x_{m_1},x_{m_2}$ are linear and $x_i,x_j,x_k$ are 2-occurring;
            \item $\mathbf{S} \models x^2y^2 \approx y^2x^2 \approx (xy)^2 \approx (yx)^2$.
        \end{itemize}
        By deletion, we must also have that if $Y \subset X_n$ with $\lvert Y \rvert < 5$, then $\mathbf{S} \models \mathbf{s} \approx \mathbf{v}[Y]$. It follows that for any $i,j \in \underline{n}$ and any subset $Y \subset X_n$ with $\lvert Y \rvert \leq 5$, we obtain $\mathbf{S} \models \mathbf{t}_{ij}[Y] \approx \mathbf{v}^{(ij)}[Y]$. Therefore $\mathbf{S} \models \mathbf{t}_{ij} \approx \mathbf{v}^{(ij)}$ by Lemma \ref{lemma_asabtb} and the claim is proved. 
\end{proof}
\section*{Acknowledgements}
The author would like to express his deepest gratitude to the referee. Their  comments and suggestions have significantly improved this paper. 
\bibliography{refs}       
\bibliographystyle{abbrv} 
\end{document}